\documentclass[12pt]{amsart}
\usepackage{amsthm,amsmath,amssymb,amscd}
\usepackage{amssymb}
\usepackage{graphics}

\usepackage{color}

{
       \newtheorem{theorem}{Theorem}
       \newtheorem{proposition}[theorem]{Proposition}

       \newtheorem{lemma}[theorem]{Lemma}
      \newtheorem{corollary}[theorem]{Corollary}

{\theoremstyle{definition}

       \newtheorem{definition}[theorem]{Definition}
       \newtheorem{remark}[theorem]{Remark}
}

\newcommand{\RR}{{\mathbb{R}}}

\newcommand{\CC}{{\mathbb{C}}}

\newcommand{\PP}{{\mathbb{C} \mathbb{P}}}

\newcommand{\cA}{{\mathcal A}}

\newcommand{\Si}{\Sigma}
\newcommand{\uC}{\underline C}

\newcommand{\Rm}{\operatorname{Rm}}
\newcommand{\p}{\partial}

\newcommand{\n}{\nabla}

\newcommand{\la}{\langle}
\newcommand{\ra}{\rangle}

\newcommand{\der}[1]{\frac{\partial}{\partial #1}}
\newcommand{\fder}[2]{\frac{\partial #1}{\partial #2}}

\newcommand{\il}[1]{\int\limits_{#1}}

\newcommand{\db}{\overline{\partial}}

\begin{document}

\title [Positive isotropic curvature: macroscopic dim. and $\pi _1(M)$]
 {Macroscopic dimension and fundamental group of manifolds with positive
isotropic curvature}

\author[Gabriele La Nave] {Gabriele La Nave}
\address{lanave@illinois.edu\\  University Illinois at Urbana-Champaign}

\vspace{-,15in}
\begin{abstract}
We prove a conjecture of Gromov's to the effect that manifolds with
isotropic curvature $K^{isotr}_{\CC} (M) \geq \epsilon ^{-2}$ and with bounded geometry
 are macroscopically $1$-dimensional on the scale
 $>>\epsilon$. As a consequence we prove that compact manifolds with
 positive isotropic curvature have virtually free fundamental groups. Our main technique is modeled on Donaldson's version of H\"ormander technique to produce (almost) holomorphic sections  which we use to construct destabilizing sections.
\end{abstract}
\maketitle
\section{Introduction}

Given a Riemannian manifold $(M,g)$, one can extend the metric tensor
in two ways to the complexified tangent bundle $TM\otimes \CC$: as a
complex bilinear $( \cdot \, ,\, \cdot )$ form or as a Hermitian form $\la \cdot \, ,\, \cdot \ra_{\CC}$.
A tangent vector $v$ is called {\it isotropic} if $(v,v)=0$, and
analogously a 2-plane $\pi \subset TM\otimes \CC$ is called {\it
  totally isotropic} if $(u,v)=0$ for every $u,v \in \pi$.

If we view the Riemannian curvature tensor $R$, as a {\it quadratic form} $\Rm$
on $\bigwedge ^2 TM$ (this is the {\it curvature operator}), we can
clearly extend it to a quadratic form $\Rm _{\CC}$ on $\bigwedge ^2 (TM
\otimes \CC).$ 

The {\it sectional curvature} $K(u,v):=\la R(s,u)u ,s \ra$, inasmuch as the restriction of $\Rm$
on bivectors, can then be extended to the complexified tangent bundle.
In other words, if we think of the sectional curvature as a function
$K$ on $Gr (2, TM)$-- the Grassmannian bundle of ~2-planes in $TM$-- we can
extend it as a function $K_{\CC}$ to the Grassmannian bundle of
complex ~2-planes in $TM\otimes \CC$
$Gr _{\CC}(2, TM\otimes \CC)$, as follows: 
$$K _{\CC} (\pi):=\Rm (v\wedge w, \overline {v\wedge w})$$
where $v$ and $w$ are two vectors in $\pi$ which are orthogonal with
respect to the Hermitian product $\la . , .\ra _{\CC}$.
 Restricting the function $K_{\CC}$ to the subbundle of {\it totally isotropic} 
   two-planes 
$Gr_{\CC} ^{isotr} (2, TM)$ (which is non-empty only if
 $dim(M)\geq 4$) we obtain the {\it isotropic curvature}
 $K^{isotr}_{\CC}$. 

We are now ready for 
\begin{definition}
We say that $M$ has {\it positive isotropic curvature} if
$K^{isotr}_{\CC}>0$ and that the isotropic curvature is bounded below
by $k$ if $K^{isotr}_{\CC}\geq k$ 

\end{definition}
\noindent
These conditions are readily seen to be equivalent to the requirement that for any {\it orthonormal} ~4-frame $\{e_1,\, e_2,\, e_3, \,e_4\}$ one has:
$$R_{1313}+R_{1414}+R_{2323}+R_{2424}-2R_{1234}>0\;\; \text{ resp. }  >k$$
where $R_{ijkl}= {\rm Rm} (e_i,e_j, e_k,e_l)$.

Positivity of the isotropic curvature $K^{isotr}_{\CC}$ is implied by
the positivity of the complex sectional curvature $K_{\CC}$ (which in
turn is implied by the strong condition of positivity of the
Riemannian curvature operator), and it implies the positivity of the
{\it scalar curvature}.

 Nonetheless there are examples of manifolds
with positive sectional curvature, or even positive {\it Ricci
  curvature}, which do not admit positive isotropic curvature, and
more importantly there are examples of manifolds which admit positive
isotropic curvature, but which cannot
admit a metric with positive Ricci curvature. For instance, in
dimension $n=4$, every locally conformally flat manifold with positive
scalar curvature has a metric with positive isotropic curvature, as
shown in \cite{mw}. More explicitly, $M_k:=\#_{i=1}^k S^3\times S^1$ has a
metric with positive isotropic curvature for any $k$ (indeed, as
proved in \cite{MW}, the connected sum of any two manifolds with
uniformly positive isotropic curvature admits such a metric), but for
topological reasons (which are a consequence of the splitting theorem of
Cheeger-Gromoll), it cannot support a metric with {\it non-negative
  Ricci curvature}. On the other hand $\PP ^2 {\Large{\#}}_{i=1}^k \overline \PP^2$
admits no metric such that $K^{isotr}_{\CC}\geq0$ for $1\leq k\leq 8$
(cf. \cite{mw}),
nonetheless it admits a metric with $Ric >0$ by Yau's proof of Calabi's conjecture (as it is easy to show that the {\it canonical class} is positive); in fact, even a {\it
  K\"ahler-Einstein} one (necessarily with positive scalar curvature) if $k>2$ (cf. \cite{Tian}).

 The main result in \cite{mm} (cf. \cite{bsch} for a nice survey) is the following strong topological restriction
 imposed by positive isotropic curvature:
\begin{theorem}\label{theorem:mm}{\bf (Micallef-Moore)}
 A compact manifold $M$ with positive
 isotropic curvature is such that $\pi _i(M)=(0)$ for $2\leq i\leq
 [\frac{n}{2}]$, where $[x]$ indicates the integral part of $x$.
\end{theorem}

The connected sum of manifolds with $K^{isotr}_{\CC}\geq k>0$ also
admits such a metric (cf. \cite{MW}).  $S^1\times S^n$ and
spherical space forms admit such metrics. The main conjecture for
manifolds with positive isotropic curvature is that, so far as the
fundamental group is concerned, these connected sums are the only
manifolds with positive isotropic curvature. More precisely, the following conjecture has
been put forth:

\vskip 5pt
\noindent
{\bf Conjecture.} {\it The fundamental group of a compact
manifold with positive isotropic curvature is {\em residually free}}.

A particular case of this conjecture, to the effect that the
fundamental group of such manifolds cannot contain subgroups
abstractly isomorphic to the fundamental group of a Riemann surface of
genus $p\geq 1$, has been proved by A. Fraser--J.Wolfson
(cf.\cite{FW}) (the special case $p=1$ had been
previously proved by A. Fraser in the foundational paper \cite{F2}).

Following Gromov (cf. \cite{g}) we define:

\noindent
{\bf Definition.} \label{macroscopic} A metric space $V$ has {\it
macroscopic dimension at most $k$ on the scales $>> \epsilon$ },
if there exists a $k$-dimensional polyhedron $P$ and a continuous
map $\phi : V \to P$, such that for every fiber $\phi ^{-1} (p)
\subset V$, $diam \; \phi ^{-1}(p)\leq \epsilon$. If this is the case,
one writes $dim _{\epsilon} V\leq k$. Moreover, one sets:
 $$dim _{\epsilon} V := \sup \{k \; : \; dim_{\epsilon} V\leq k\}$$

The prospective of Gromov's in \cite{g} on manifolds with positive
curvature through his concept of {\it macroscopic dimension} turns out
to be particularly fruitful  in answering the above mentioned
conjecture, especially when combined with the ideas
of Gromov-Lawson (cf. \cite{gl}) and the stability inequality due to
Micallef and Moore (cf. \cite{mm}).

In fact, we can show the following conjecture of Gromov's
(cf. \cite{g}, para. 3):
\begin{theorem}\label{main}   
 If $(M,g)$ is of bounded geometry and $K^{isotr}_{\CC} (M) \geq \epsilon ^{-2}$, $dim (M) \geq 4$.
 Then $M$ is macroscopically $1$-dimensional on the scale
 $>>\epsilon$.
\end{theorem}
Gromov remarks that this is enough to show that $\pi _1 (M)$ is residually free. This in fact can be done by applying the estimates to the universal covering, but we will resort to a different approach due to Ramachandran-Wolfson (cf. \cite{rw}), as explained below.

The assumption that the manifold has bounded geometry is necessary for the solution of the Plateaux problem.
One should remark that the assumption that the isotropic curvature be
strictly (and uniformly) positive cannot be relaxed, as shown by the
example $M=\Si \times S^k$, which admits a metric with {\it
  non-negative} isotropic curvature-- here $\Si$ is a Riemann
surface of genus $p\geq 2$ and $S^k$ is the round $k$-sphere. This is reflected in the fact that the stability inequality, among other things, yields no information when the isotropic curvature is not strictly (and definitely) positive.

One of the main ingredients--just like in Fraser's fundamental paper \cite{F2}-- is the second variation formula of the energy functional, or better
yet its manifestation in the form of the {\it stability inequality}
(cf.eq. (\ref{equation:stability})), as described in \cite{mm}, which
we will explain briefly in section \ref{secondvariation}.

Another major --and arguably more important for this paper-- ingredient comes into play in the construction of the {\it
  destabilizing} sections and is based on using Donaldson's version of H\"ormander techniques to construct suitable {\it almost holomorphic} sections (cf. \cite{D} and \cite{DS}). Our construction is closely based on Donaldson's (cf. also Donaldson-Sun), with the difference that we need to make sure that the section also be {\it isotropic}. The philosophy is mostly based on Donaldson's construction of almost holomorphic sections in \cite{D}, the major difference being that unlike  in \cite{D}, we work with integrable complex structures, as in \cite{DS}, but even in this case we do not need the full blown H\"ormander technique as we merely need the sections to be almost holomorphic (and in fact we eventually need to multiply them by cut-off functions, in order to make them compactly supported).  We thus effectively first construct highly peaked {\it holomorphic sections} (i.e., concentrated on a very small ball) following Donaldson's argument (the result resembles Tian's construction of peak sections, cf. \cite{Tian3} but it is remarkably different).
  
  Finally we use a cut-off function argument to render the sections thus constructed compactly supported.
 The estimate we prove using these ingredients -- and itself the major ingredient in proving Theorem \ref{main}--is 
 \begin{theorem}\label{main2}
 Assume that $K^{isotr}_{\CC} (M) \geq \epsilon ^{-2}$, and that $dim (M) \geq 4.$
Let $f:D \to M$ be a stable, minimal (possibly branched) immersion. Then for every point $p\in D$ there exists a smooth compactly supported {\it
  isotropic}  section
$\sigma =\sigma _p $ of
$E$, and a constant $C$ such that:
\begin{equation} \label{testsections} \frac{\il{D}|\nabla _{\frac{\partial}{\partial \bar z}}\sigma |^2
dV}{\il{D}|\sigma |^2 dV}\leq K\frac{1}{r^2} .\end{equation}
where $r:= \rm{dist}_{D} (p, D).$
Furthermore, the constant $K= K(n)$ is computable and it can be taken to be equal to $\frac{9^3\,n\,\pi}{4}$.

 \end{theorem}
 
 \begin{corollary}
  Assume that $K^{isotr}_{\CC} (M) \geq \epsilon ^{-2}$, and that $dim (M) \geq 4$ and that $M$ has bounded geometry. Then for every closed curve $\gamma$ such that $[\gamma]=0$ in 
  $H_1(M)$, then:
  $$\rm{ Fill Rad} \gamma\leq C \epsilon$$
 where we can take $C=\sqrt{\frac{9^3\,n\,\pi}{4}}$ \end{corollary}
 An immediate corollary of this coupled with Theorem 1.2 in \cite{rw}, is the following
 \begin{theorem}
 If $M$ is a closed manifold such that that $K^{isotr}_{\CC} (M) \geq \epsilon ^{-2}$ then $\pi _1(M)$ is residually free.
 \end{theorem}
 Also, using the main theorem of \cite{sesh} one has the following immediate consequence:
 \begin{theorem}\label{homol-schoen}
 Let $M$ be a closed, orientable Riemannian n-manifold with positive
isotropic curvature of dimension $n\geq 5$. Then there exists a finite cover of $M$ which is homeomorphic to the connected sum of $k$ copies of $S^{n-1}\times S^1$.
 \end{theorem}
Finally, we would like to point out that in dimension ~4 something much stronger is true:  using the
Ricci flow, R. Hamilton (cf. \cite{H}) and B.-L. Chen and X.P. Zhu
(cf.\cite{cz}) have been able to prove that a compact ~4-manifold with
positive isotropic curvature and containing no essential incompressible
~3-dimensional space form, is {\it diffeomorphic} to $S^4$,
$S^3\times S^1$, $\RR {\mathbb P}^4$ and $S^3\tilde {\times} S^1$ and
their connected sums (naturally, the last two do not occur in the oriented case). 
In fact in a recent beautiful paper,  Chen, Tang and Zhu (cf. \cite{ctzhu})-- using Hamilton's Ricci flow--
have proven the ~4-dimensional version of a very far reaching conjecture due to R. Schoen, which claims that the (much stronger) differential version of Theorem \ref{homol-schoen} should hold.

We finally would like to end the introduction by pointing out an extremely interesting question posed by M. Gromov (cf. \cite{g2}) on possible generalizations to the singular setting:``Is there a natural class of singular spaces $X$ with $K_\CC(X)>0$ that would satisfy (a suitable version of) the Micallef-Moore and/or La Nave  bounds on indices and sizes of harmonic maps of surfaces into $X$?"

\subsection{Outline of proof}

As explained in the introduction, the main ingredients in the proof are the stability inequality of Micallef-Moore (already used also in \cite{F2} for similar purposes) and the construction of almost holomorphic sections with controlled $L^2$-norms. 
 The construction of the sections roughly goes as follows. 
 We first show in Proposition \ref{holomorphic-isotropic} that we can reduce to and solve a Dirichlet problem, when the Hermitian metric $H$ and the metric on the disk are controlled in $L^{\infty}$ with respect to, respectively, the standard Hermitian metric $H_0(v,w)= \sum _i \; v_i\bar w_i$ and the flat metric $g_0= dx^2+dy^2$.
 
 Next in section \ref{model-gaussian} we show, using the Dirichlet problem solved in Proposition \ref{holomorphic-isotropic},  that in the model situation we can find the sections we seek. By means of rescaling, we then prove in Proposition \ref{holo-isotrop} that in the case in which the metric on the disk is the flat metric $g_0= dx^2+dy^2$, we can reduced to the aforementioned model case, with controlled errors --thanks to Proposition \ref{controlforholsection}, which is a direct application of one of the incarnations of the Bochner formula (as explained in \cite{DS} in the line bundle case).
  Then the use of rescaling allows us to construct the desired Gaussian holomorphic {\it isotropic} sections Theorem \ref{loca-model} with controlled $L^2$ norms. Finally we use a cut-off function argument in Proposition \ref{destabsection} to construct the "almost-destabilizing" sections: smooth almost-holomorphic {\it compactly supported} isotropic sections with controlled $L^2$-norms on concentric balls.

\section*{Acknowledgments} The author is very grateful to T. Colding for having introduced him to the problem and to Harish Seshadri for pointing out the validity of Theorem \ref{homol-schoen} (and the relevant reference \cite{sesh}) and to Xi-Ping Zhu for pointing out the reference \cite{ctzhu} and in particular Schoen's conjecture.
I would also like to thank G. Tian for constant support.
\section{The second variation formula}\label{secondvariation}

Let $f:\Sigma\to M$ be a stable minimal surface in $M$.
Consider the pull back of the tangent bundle with the pull back of the metric and (resp. normal)
connection $\nabla$. Let $E=f^*TM\otimes
\CC$  be the complexified
bundle. The metric on $f^*(T _M)$ extends as a complex bilinear
form $( \cdot , \cdot )$ or as a Hermitian metric $\la \cdot ,
\cdot \ra$ on $E$, and the connection $\n$ (the pull-back via $f$ of the Levi-Civita connection of $M$) and curvature
tensor extend complex linearly to sections of $E$. Moreover the 
connection is Hermitian with respect to $\la \cdot , \cdot \ra$.

By a well known theorem, (cf. \cite{km} and \cite{ahs}), there is
a unique holomorphic structure on $E$ such that the $\bar{\p}$
operator
$$
      \bar{\p}: {\cA}^{p,q}(E) \rightarrow {\cA}^{p,q+1}(E),
$$
\noindent
where ${\cA}^{p,q}(E)$ denotes the space of $(p,q)$-forms on $\Si$
with values in $E$, is given by
$$
      \bar{\p}\omega=(\nabla_{\der{\bar{z}}} \omega)d\bar{z}
$$
\noindent
where $\der{\bar{z}}=\frac{1}{2}(\der{x}+i\der{y})$, for local
coordinates
$x, y$ on $\Si$. 

One can choose a metric $ds^2= \lambda (dx^2+ dy^2)$ on $\Si$ compatible with the conformal structure determined by the pull-back of $G=f^*g$ of the metric $g$ on $M$ via the immersion $f$. Then one has:
$$\Vert \nabla_{\der{\bar{z}}} W\Vert ^2 dV _{ds^2}= |\nabla_{\der{\bar{z}}} |_H^2 dx\wedge dy$$
where by $|\nabla_{\der{\bar{z}}} |_H^2$ we mean $\la \nabla_{\der{\bar{z}}} , \nabla_{\der{\bar{z}}}\ra$, the Hermitian scalar product induced by $f^*g$ on $E$.
Suppose $f:\Sigma\to M$ is a {\it stable} minimal immersion. Then the complexified stability
inequality (see \cite{SiY}, \cite{Mi}, \cite{mm}, \cite{bsch} and \cite{F2}) for
the {\it Energy functional} reads:
$$\il{\Si} \la R(s,\fder{f}{{z}})\fder{f}{\bar{z}},s \ra \; dx\wedge dy
    \leq \il{\Si} | \nabla_{\der{\bar{z}}} s|^2
    \; dx\wedge dy$$ for all $s \in \cA ^{\infty} _0(E)$, the space
    of smooth sections of $E$ with compact support (if the surface has
boundary) or for all $s \in \cA ^{\infty}(E)$, the space of smooth
sections, if $\Si$ is closed. Assume now that $s$ is {\it isotropic}.
Since $f$ is conformal, $\fder{f}{z}$ is isotropic and
$\{s,\fder{f}{z}\}$ spans an isotropic two-plane. If the isotropic
curvature is such that $K^{isotr}_{\CC}\geq \epsilon ^{-2}$, we
get:
\begin{equation} \label{equation:stability}
    \epsilon ^{-2} \il{\Si} |s|^2 \; dV
    \leq \il{\Si}|\nabla_{\der{\bar{z}}}s|^2   \; dV
\end{equation}
where $dV$ denotes the area element for the induced metric $f^*g$
on $\Si.$
Here the norms are
$\frac{1}{\lambda ^2}$ times the corresponding norms coming from
$TM\otimes\CC$ (cf.\cite{mm}).
We will also denote by $N:= \nu_f\otimes \CC$ where $\nu _f$ is the normal bundle of $f$, i.e. the bundle defined by the exact sequence of real bundles:
$$0\to T_\Sigma \to f^* TM \to \nu _f\to 0.$$
The same considerations we did for $E$ hold for $N$ and as observed by A. Fraser in \cite{F}, the stability inequality can be formulated as:
\begin{equation} \label{equation:stability-2}
    \epsilon ^{-2} \il{\Si} |s|^2 \; dV
    \leq \il{\Si}|\nabla_{\der{\bar{z}}}^{\perp}s|^2   \; dV
\end{equation}
for any compactly supported section $s$ of $N$ and here $\nabla ^\perp$ is the connection induced to the normal bundle from the Levi-Civita connection.

\section{The test sections}

Throughout this section, $f:D\to M$ will be a stable, minimal
(possibly branched) {\it proper} immersion from the {\it disk} D to M.
We will also maintain the notation of section \ref{secondvariation}:
$E:= f^*(TM\otimes \CC)$ etc.
\noindent
Let $q_D$ be the quadratic form on $E$ induced from the
$\CC$-bilinear form $(,)$. If $\gamma $ is a smooth curve in $D$,
then we denote by $q_{\gamma}$ the restriction of $q_D$ to ${E
_\mid} _{\gamma}$. Furthermore we will call a smooth section
$\alpha$ of ${E_\mid} _{\gamma}$ {\it isotropic} if
$q_{\gamma}(\alpha, \alpha )=0$.

\subsection{Curvature of Hermitian metrics on Riemann surfaces}\label{curvature}
Recall that given a Hermitian holomorphic bundle $(E,H)$ one has a unique connection $\nabla $ whose $(0,1)$-part $\nabla^{(0,1)}$ is equal to $\bar \p$ (the operator determining the integrable complex structure of $E$): the Hermitian connection. 
If $\{ e_1,\cdots, e_n\}$ is a {\it holomorphic frame}, and if:
$$h_{i\bar j} := H(e_i,\bar e_j)$$
then the connection $1$-form and the curvature ~2-form (which is an $End(E)$ valued ~2-form) are given (respectively) by: 
\begin{equation}\label{connect-curv} A = \p h\cdot h^{-1},\;\; \; \Theta _H:= \bar \p H= - \p\bar \p h \cdot h^{-1}+ \p h \cdot h^{-1} \wedge \bar \p h \cdot h^{-1}\end{equation}
The curvature tensor-- which is a section of $E^*\otimes \bar E^*\otimes \Omega_N^{(1,0)}\otimes  \Omega_N^{(0,1)}$ where $\Omega_N^{(p,q)}$ is the space of $(p,q)$-forms-- is given by:
$$R(H)(v, w, s,\bar t):= H( \Theta _H(s\wedge \bar t) v, \bar w)$$
and in the frame  $\{ e_1,\cdots, e_n\}$ and (local) holomorphic coordinates $z_1, \cdots , z_m$ on $N$:
\begin{equation}\label{hermitiancurvature}R_{i \bar j \alpha \bar \beta }:= R(e_i, \bar e_j, \frac{\p}{\p z^\alpha}, \frac{\p}{\p\bar z^\beta})= - \frac{\p ^2 h_{i\bar j}}{\p z^\alpha \p \bar z^\beta}+ \frac{\p h_{i\bar t}} {\p z^\alpha} h^{s\bar t}  \frac{\p h_{s\bar j}} {\p \bar z^\beta}\end{equation}

\noindent
On a Riemann surface, having fixed an holomorphic coordinate $z$, we simply denote the curvature:
\begin{equation}\label{hermitiancurvature-dim1}R_{i \bar j  }:= R(e_i, \bar e_j, \frac{\p}{\p z}, \frac{\p}{\p\bar z})= - \frac{\p ^2 h_{i\bar j}}{\p z \p \bar z}+ \frac{\p h_{i\bar t}} {\p z} h^{s\bar t}  \frac{\p h_{s\bar j}} {\p \bar z}\end{equation}

\noindent
so that:
$${\Theta _H}_j ^i=h^{i\bar s} R_{i\bar s} \;  dz\wedge d\bar z.$$
One can define th {\it
 Ricci curvature} of $h$, denoted ${\rm Ric}(h),$ as
follows (cf. \cite{w} ch.3 or \cite{K} section 1 for this notion). Let $\{e_{\sigma}\}$ be a holomorphic frame for $E$ and $\{
V_i\}$ any frame field of type $(1,0)$ relative to some fixed
Hermitian metric $g$ on $N$; then set:
\begin{equation}\label{riccicurvature} {\rm Ric}_h(V_i, V_j):=\sum_{i, \nu} h^{\alpha \bar \beta} R(e_\alpha, \bar e_\beta, V_i, {\bar V}_j)\end{equation}
where $h_{i\bar j}:= h(e_{\nu},  e_{\xi})$.
In components, this is simply:
$$R_{\alpha \bar \beta}:= {\rm Ric}(\frac{\p}{\p z^\alpha}, \frac{\p}{\p\bar z^\beta})= h^{i\bar j } R_{i \bar j\alpha \bar \beta }.$$

\noindent
We also set:
\begin{equation}\label{contractedcurvature}K_{i\bar j} (g,h):=g^{\alpha \bar \beta}  R_{i \bar j\alpha \bar \beta }.\end{equation}

\noindent
which is referred to as the {\it mean curvature form} of $(E,h)$ over $(N,g)$ (cf. \cite{K} section 1).
We observe that if $N=\Sigma$ is a Riemann surface,  then the tensor $K_{i\bar j}$ and the full curvature tensor $R_{i \bar j \alpha \bar \beta }$ for a holomorphic Hermitian vector bundle $(E,h)$ are equivalent.
 Note that the notion of Ricci curvature coincides with the standard notion of
Ricci curvature, in case $E$ is the tangent bundle of $N$.

One has readily:
\begin{lemma}\label{lemma-curvature}
Let $(M,g)$ has isotropic curvature bounded from below by $C_I$ (not necessarily positive).
Let $f: D\to M$ be a minimal  immersion and let $E:=f^*TM \otimes \CC.$
 then for every holomorphic isotropic section $\sigma$ of $E$ such that $\sigma \wedge \frac{\p f}{\p \bar z}\neq 0$ one has:
\begin{equation} \label{curv-iso-subbundle} R (H)(\sigma, \sigma) := R(H) (\frac{\p}{\p z} , \frac{\p}{\p\bar z}, \sigma, \bar \sigma) \geq C_I \, \lambda \,\Vert\sigma\Vert^2  \end{equation}
where $g= \lambda \left( dx^2+dy^2\right)$ is the (necessarily K\"ahler) metric on $D$ induced --via $f$-- from $(M,g)$.

Moreover, if $(M,g)$ be a Riemannian manifold with ${\rm Ric}(g)\geq -C$ (resp. ${\rm Rm}(g)\geq -C$). Given a minimal immersion $f:D\to M$, let $E=f^*TM \otimes \CC$ with the induced holomorphic structure and hermitian structure $H$. Then the Ricci curvature (resp. curvature tensor) of $(E,H)$:
$${\rm Ric(H)} \geq -C\qquad \text{ (resp. } {\rm Rm}(h)\geq -C\text{)}$$
 \end{lemma}
\begin{proof}
The only non-trivial components of the curvature of the Hermitian metric $H$ on $E$ are, for any section $W$ of $E$:
$$R(H) (\frac{\p}{\p x}, \frac{\p}{\p y})W=\left( \nabla _{\frac{\p}{\p x}}\nabla _{\frac{\p}{\p y}}- \nabla _{\frac{\p}{\p y} }\nabla _{\frac{\p}{\p x}}\right) W$$
but in terms of the curvature of $g$ these are equal to:
$$R_g (\frac{\p}{\p x}, \frac{\p}{\p y})W$$
evaluated along $f(D)$.
Finally, equation \eqref{curv-iso-subbundle} is a simple consequence of the computation above and the definition of isotropic curvature on the isotropic plane $\sigma \wedge \frac{\p f}{\p \bar z}\neq 0$. Alternatively, one can see that the $\CC$-bilinear extension of ${\rm Rm}$ is the curvature of the complex  connection obtained by complexifying the Levi-Civita connection.  \end{proof}
\begin{remark}
In fact, Micallef-Moore in \cite{mm} showed that $(M,g)$ has positive isotropic curvature if and only if ${\rm Rm}_\CC(v,\bar v, w, \bar w)>0$ for every $v, w\in TM\otimes \CC$ such that $g(v,v)= g(v,w)=g(w,w)=0$, which are $\CC$-linearly independent (cf. also Proposition 7.2 in \cite{b})
\end{remark}

Also a very important standard fact is the following (cf. \cite{gh} pg. 78-79):
\begin{lemma}\label{positivityofquotients}
Let $(E,H)$ be a holomorphic Hermitian vector bundle and $\pi:E\to Q$ a holomorphic quotient bundle endowed with the quotient Hermitain metric $H_Q$.  Let also $F\subset E$ an holomorphic sub bundle such that the quotient $E/F\simeq Q$ and $H_F$ the induced Hermitian metric
Then for the curvature operator:
$$\Theta (H_Q)= \Theta (H) \mid _Q + S\wedge S^* \text{ and } \;\; \Theta (H_F)= \Theta (H)\mid _F - S\wedge S^* $$
where $S= \nabla _E - \nabla _F$ is the {\it second fundamental form} of $F$ --here $\nabla _E$ and $\nabla _F$ are the metric connections of $(E, H)$ and $(F, H_F)$ respectively.
In particular:
$$\Theta (H_Q)\geq  \Theta (H) \mid _Q.$$

\end{lemma}

\subsection{Bochner technique in Complex Differential Geometry}

 When working with a vector bundle $E$ we will often use the norms defined by the rescaled metrics $R^2\,g$, which has the effect of rescaling lengths by $R$ and  volumes by $R^{2n}$.
 
  We will use the notation $g_R:=R^2\,g$ to denote the rescaled metric and we will further simplify notation by writing $dV_R$ for $dV_{g_R}$, the corresponding volume form. Then the scaling weight gives
\begin{equation}  \Vert \nabla f \Vert_{L^{2}(dV_R)}=R^{n-1}  \Vert \nabla f \Vert_{L^{2}(dV_g)}\;  \text{ and } \; \Vert f \Vert_{L^{2}(dV_{g_R})}=R^n\, \Vert f \Vert_{L^{2}(dV_g)}. \end{equation}

We will make use of the following various forms of Laplacian operators:

 $$\Delta_{\db}= \db^{*}\db+ \db \db^{*},$$
 with adjoints defined using the $L^2$-metric induced from $g$
$$\Delta_{R,\db}= \db^{*}_R\db+ \db \db^{*}_R,$$
with adjoints defined using the rescaled metric $g_R$.
\noindent

Given a Hermitian holomorphic vector bundle $(E,H)$, the Laplace-Beltrami operator associated to the
connection $\nabla$ is
$$\Delta =\nabla \nabla^*+\nabla ^* \nabla$$
where
$$\nabla^*:C^{\infty}(M,\Lambda^qT^\star_M\otimes E)\to
C^{\infty}(M,\Lambda^{q-1}T^\star_M\otimes E)$$
is the (formal) adjoint of $\nabla$ with respect to the $L^2$ inner product.

If $M$ is a compact complex manifold
equipped with a hermitian metric $\omega=\sum\omega_{jk}dz_j\wedge d\bar
z_k$ and $E$ is a holomorphic vector bundle on $M$ equipped with a
Hermitian metric, and let $\nabla=\nabla ^{1,0}+\nabla^{0,1}$ be its Chern curvature form (that is to say $\nabla^{0,1}= \bar \p$), decomposed in $(1,0)$ and $(0,1)$ parts respectively.
In the same way one forms the  Laplace-Beltrami operator $\Delta$, one can form the following complex Laplace operators:
$$\Delta'=\nabla^{1,0}\nabla^{^{1,0}\star}+\nabla^{^{1,0} \star}\nabla^{1,0},\qquad
\Delta''=\nabla^{0,1}\nabla^{0,1 \star}+\nabla^{0,1 \star}\nabla^{0,1}.$$

\noindent

Remark that if the connection is Hermitian, then $\nabla ^{0,1} = \bar \p$ and thus:
$$\Delta ''= \bar \p \bar \p ^* +  \bar \p ^* \bar \p $$
Let also $\Lambda$ be the adjoint of the operator (known as Lefschetz operator) $L=(\omega \wedge )\otimes 1$, defined on $\Omega ^{p,q}M\otimes E$, with respect to the Hermitian product on $\Omega ^{p,q}M\otimes E$ induced by $g$ and $H$.
The main identity we will be using is (cf Corollary 1.4.13 in \cite{mama}):
\begin{theorem}[ Bochner-Kodaira-Nakano identity]\label{bochner-kodaira} If $(X,\omega)$ is
K\"ahler, the complex Laplace operators $\Delta'$ and $\Delta''$ acting
on $E$-valued forms satisfy the identity
$$\Delta''=\Delta'+[\sqrt{-1} \Theta(h),\Lambda].$$

\end{theorem}
Another important piece of information is the following (cf. \cite{Dem}, Ch. VII, eq. (7.1)): 
\begin{lemma}
If $E$ is a Hermitian vector bundle on a Riemann surface $\Sigma$, then for any smooth section $s$ of $E$, one has:
$$[\sqrt{-1} \Theta(h),\Lambda]s= K(g,h)(s)$$
where $K(g,h)$ is defined in equation \eqref{contractedcurvature}. Furthermore the $\Delta ''$ acting on sections $\phi$ of either $\Omega ^{p,0}_\Sigma$ or  $ \Omega ^{1,q}_\Sigma$ is:
$$\Delta '' \phi =-\nabla _Z \nabla _{\bar Z} \phi$$
where $Z$ is a (local) holomorphic frame of $T^{1,0} \Sigma$. In particular for a smooth section $\sigma$ of $\Omega ^{p,0}_\Sigma(E)$ or of $ \Omega ^{1,q}_\Sigma(E)$ one has:
\begin{equation} \begin{aligned}\label{2D-bochner}\Delta '' \sigma& =  \sum _{i=1}^n \left(-\nabla _Z \nabla _{\bar Z} \sigma ^i \right) e_i - \sum _i \left( \nabla _Ze_i\right) \left( D_{\bar Z} \sigma ^i\right) \\&+ \sum _{i, j} e_j \left( i_Z \Omega ^j _i \wedge i_ {\bar Z} \sigma ^i\right) \end{aligned}\end{equation}
where $\{ e_i\}$ is a holomorphic frame for $E$ and $\sigma = \sum _i \sigma ^i \, e_i$.
 \end{lemma}
\begin{proof}
This is a consequence of the definition of the operator $\Lambda$ which in particular implies that if $Z$ is a local orthonormal frame of $T^{(1,0)}\Sigma$, then (cf. \cite{mama} section 1.4.3 and in particular formula (1.4.63) therein)
$$\Lambda = -\sqrt{-1} i_Z i_{\bar Z}$$
thus for any section of  $\Omega ^{p,q}M\otimes E$:
$$[\sqrt{-1} \Theta(h),\Lambda] s= R(H) (Z, \bar Z) \bar Z \wedge i_{\bar Z} s$$
and the fact that on a Riemann surface $\Sigma$ the only nontrivial $\Omega^{p,q}(E)$'s are either of the form $\Omega ^{p,0}(E)$ or $\Omega ^{n,q}$ (since $n=1$). This implies that the curvature of the background metric cancels out
\end{proof}

\subsection{Rescaling}\label{rescaling}

As before, we are given $(f:D\to M, E, H)$ a triple consisting of a stable minimal immersion $f:D\to M$, the induced vector bundle $E:=f^*TM \otimes \CC$ with induced complex structure and with induced Hermitian metric $H$, and connection $A_E$. Also, $D$ is endowed with the pull-back metric $g_D= \lambda ^2 (dx^2 +dy^2)$ (necessarily K\"ahler by virtue of dimension).

\noindent
 We will actually endow $E$ with the tensor product metric $H_R :=   H\otimes e^{-R^2|z|^2} H_0$ (here $H_0$ is the Euclidean Hermitian metric on $E$) and the tensor product connection:
 $A_R:=A_E\otimes R^2 A$

\noindent
We rescale the background metric $g_D= \lambda ^2 (dx^2 +dy^2)$ and the (holomorphic) homotetic transformation, for any point $z_0\in D$:
$$ \Phi _{R,z_0}: D_R\to D, \text{ defined by } \;  \Phi _{R,z_0} (z) := z_0+ \frac{z}{r}$$
and pulling back all the geometric quantities:
$$\begin{aligned} &g_{D_R} := \Phi _{R,z_0} ^* g_D; \;\; E_R:=  \Phi _{R,z_0}^* E;\;\; H_R:=  \Phi _{R,z_0}^* \left( H\right); \;\; \\&A_R= \Phi _{R,z_0} ^*  \left(A_E\otimes R^2A\right).\end{aligned}$$

The map $\Phi _R$ and the rescaling associated with it make sense for any vector bundle $E$, and it is this type of rescaling we will be most interested in.
The main elementary observation about rescaling that shall be used in the proof of Theorem \ref{main2} is the following:
\begin{lemma}\label{rescale} 
Let $\tilde g_R:= R^2\, g_{D_R}$. Then 
for any $C^1$ section $\sigma$, one has:
$$\il{D_R}|\nabla^{\tilde g_R} _{\frac{\partial}{\partial \bar z}}\sigma_R |^2 dV_{\tilde g_R}= \il{D}|\nabla^g _{\frac{\partial}{\partial \bar w}}\sigma |^2dV_{g}$$
where $\sigma _R:=\Phi _R^*\sigma$, $\nabla ^{g_R}$ and $\nabla ^g$ are resp. the Levi-Civita connections of $g_R$ and $g$ and $w$ is a holomorphic coordinate on the unit disk $D$.
\end{lemma}
\begin{proof}
The proof-- which boils down to showing that $ \il{D}|\nabla^g _{\frac{\partial}{\partial \bar w}}\sigma |^2dV_{g}$ is invariant under conformal transformations-- is a consequence of the naturality of the Levi-Civita connection under pull-backs and the integration formula under pull-backs.  
\end{proof}

The advantage of this  Lemma is twofold: it allows one to replace the induced metric on any minimal immersion $f: D\to M$ by a conformal metric, e.g. the flat metric but also to rescale {\it ad libitum} without changing the $L^2$-norm of $\nabla ^{01} \sigma$.
\subsection{Tweaking the Hermitian metric so it has positive curvature}

This section is not necessary for the proof of the main theorem, but we deem it important for future uses.

For most of our constructions it will be convenient to be able to conformally change --by a uniformly controlled conformal factor--the Hermitian metric on $E$ so that the curvature becomes positive. This is achieved by:
\begin{proposition}\label{tweaking}
Let $\omega$ be a K\"ahler form on the disc $D\subset \CC$ and 
let $(E,H)$ be a Hermitian rank $n$ vector bundle over $D$ (hence necessarily trivial), and assume that its curvature satisfies:
$${\rm R}(H)\geq -\theta H,\qquad \text{  (resp.) } {\rm Ric}(H)\geq -\theta \omega$$ 
for some positive number $C$.
Then there exists a conformal Hermitian metric $H_\psi:= e^{-\psi} H$ on $E$ such that:
$${\rm R}_{H_\psi} \geq 2 \omega \qquad \text{  (resp.) } {\rm Ric}(H)\geq  2 \omega$$
that is to say the bundle $det(E)$ endowed with metric $h_\psi:=det (e^{-\psi}\,H)$ has positive curvature. Furthermore, $\psi$ can be chosen so that:
$$\Vert \psi \Vert_{C^{k,\alpha}}<C(\theta)$$
for a constant $C=C(\theta)$ depending only on $\theta$ and $\omega$. In particular the oscillation of $\psi$ is uniformly bounded.

\end{proposition}

\begin{proof}
This is based on two facts: on the one hand the elementary linear algebra fact that if $A$ is an $n\times n$ symmetric matrix, then there exists a $k\in \RR_+$ such that $A+k \,I_n>0$ (where $I_n$ is the identity $n\times n$ matrix); on the other hand the fact, discussed in section 3.1 that on a Riemann surface $R(H)$ (or rather $\Theta _H$) is of the same tensorial type as $\omega.$

Since the curvature of $H_\psi:= e^{-\psi} H$ is calculated--making use of equation \eqref{hermitiancurvature} and the fact that $-\p \bar \p \left( e^{-\psi}\right) = \left( \p \bar \p\psi-\p \psi\wedge\bar\p \psi \right)  e^{-\psi}$--as follows (given that $E$ has rank $n$):
$$\Theta (H_\psi)= \left(\Theta (H) +  \p \bar \p \psi \, H\right)e^{-\psi}$$
or in coordinates:
$$R_{i\bar j}(H_\psi)= \left( R_{i\bar j}  + \frac{\p ^2  \psi}{\p z\p \bar z} \, H_{i\bar j} \right)e^{-\psi}$$
and since $R_{i\bar j}(H_\psi)\geq -\theta H_{i\bar j}$ after tracing with respect to $H$, for the Hermitian metric $H_\psi$ to satisfy the conclusion of the theorem it is (necessary and) sufficient that:
$$ \Delta_{{\rm flat}} \psi
 =\frac{4}{n} k$$ 
 where $k$ is a  function (which without loss of generality we may assume to have a sign) such that
  $$-\theta \omega +k \omega\geq  2 \omega$$
 and $ \Delta_{{\rm flat}} \psi= 4 \p \bar \p \psi$ is the Laplacian with respect to the flat metric $ds^2= dx^2+dy^2$. 
 On the other hand, by standard elliptic theory, given any smooth $k$ and any boundary value $\rho$, we can find a smooth solution to:
 $$\left\{ \begin{aligned}  &\Delta_{{\rm flat}} \psi
 =\frac{4}{n} \,k\\ &\psi \mid _{\p B} = \rho \end{aligned} \right.$$
 As for the assertion on the oscillation, observe that by (Calderon-Zygmund) elliptic regularity:
 \begin{equation}\label{calderon}  \Vert \psi\Vert _{W^{\ell+2,p}(B )} \leq C(B, \ell) \left( \Vert \psi\Vert_{L^2} + \Vert \frac{k}{n}\Vert_{W^{\ell+2,p}(B)}+  \Vert \rho \Vert_{W^{\ell+2,p}(B )}  \right)\end{equation}
 We remark that since a disk of radius $R$, $D_R\subset \CC$ is strictly pseudo convex, one can actually choose a plurisubharmomic defining function $\chi$ (e.g. $\chi = |z|^2$) such that $\chi \mid_{\p D_R}=R$ and then by taking a suitable multiple $\chi=  e^{-C |z|^2}$, one can choose $\rho=CR$ (and in fact $\psi = C|z|^2$) in the above construction. 
 
  \end{proof}
%
\begin{remark}\label{remark:positive}
Our application of this Lemma will be to minimal immersions of the disc into our manifold $M$. Note that for any (possibly branched) minimal immersion  $f:D
\to M$ the vector bundle $E$ with the induced (Hermitian) metric
from $M$--which we indicate by $H:=f^*g$-- has Ricci curvature
depending only on the Ricci curvature of the metric $g$ of $M$; hence
${\rm Ric}(H)$ is bounded below by some constant $C_1$.

%

\end{remark}

We can also prove the following (which is of independent interest and ultimately unnecessary for the proof of the main theorems):
\begin{lemma}
Keeping notation and assumptions as in Proposition \ref{tweaking}, if one also has:
$$\vert {\rm Ric}_H\vert <C$$
then there exists a rank ~1 holomorphic line sub-bundle $L\subset E$ such that, if $h_L$ denotes the induced metric its curvature satisfies: $ \sup _{D'}R(h_L) >-C_1$ for any compactly embedded $D'\subset D$, where $C_1$ only depends on $C$ and the Sobolev constant and D'.
\end{lemma}
\begin{proof}
In order to prove this we notice that there must exist a holomorphic frame $\{ e_1, \cdots, e_n\}$ such that:
 $$R(h)_{1\bar 1}:= R(h) (e_1,\bar e_1, \frac{\p}{\p z}, \frac{\p}{\p \bar z})>-C$$
 Since the Ricci curvature of $h$ equals:
 $${\rm Ric } (h) = -h^{i\bar j} \frac{\p ^2h_{i\bar j}}{\p z \p \bar z} + h^{i\bar j}h^{s\bar t} \frac{\p h_{s\bar  j}}{\p z} \frac{\p h_{i\bar  t}}{\p\bar  z}$$
 --which is clearly an elliptic system-- it follows from the Calderon-Zygmund inequality that (cf. \cite{morrey}):
 $$\Vert  h_{i\bar j} \Vert _{W^{2,p} (D')} \leq C_Z(D')\; \Vert {\rm Ric } (h) \Vert _p  $$
 and the conclusion follows from Morrey's version of Sobolev inequality, which implies that:
 $$\Vert h^{s\bar t} \frac{\p  h _{s\bar j} }{\p z}  \frac{\p h _{i\bar t}} {\p \bar z} \Vert _{C^{0,\alpha} (D')} \leq C_S C_Z(D')\; \Vert {\rm Ric } (h) \Vert _p$$
 where $C_S$ is the Sobolev constant of the (compact) embedding $W^{2,p} \subset C^{1,\alpha}$ when $\frac{1-\alpha}{2}= \frac{1}{p}$.

\end{proof}
\subsection{Constructing holomorphic  isotropic sections}

%
In this section we fix the standard flat metric $g_0= dx^2 +dy^2$ with K\"ahler form $\Omega_0= \sqrt{-1} \, dz\wedge d\bar z$ on $D=D_R\subset \CC$, the disk of radius $R$.  We also endow $D_R$ with any metric which is $L^{\infty}$-close to $g_0$:
$$\frac{1}{2} \; g_0 \leq  g\leq 2\, g_0$$

We analyze the datum of a holomorphic vector bundle $E=D_R\times \CC^n \to D_R$ over  $D_R$ and endowed with a Hermitian metric such that:
$$H\leq \kappa H_0.$$
 where $H_0$ is the standard flat metric on the trivial rank ~n complex vector bundle $D\times \CC^n$-- i.e., $H_0(\xi, \xi) :=\sum _{i,j}\, (H_0)_{i\bar j}\, \xi _i \bar \xi _j=\sum _{i,j}\, \delta _{ij} \, \xi _i \bar \xi _j$, where $\delta _{ij}$ is the identity matrix.
This is the situation that one might achieve by rescaling a small ball centered and contained inside the unit ball.

We intend to show that we can always find {\it holomorphic} sections of $E$ --although clearly they will not have compact support, whence they will not be test sections for the stability inequality-- which are also $g_\CC$-isotropic (here $g_\CC$ is the $\CC$-linear extension of the metric corresponding to the Hermitian metric $H$) and that we can do so with controlled $L^{\infty}$ and $W^{1,\infty}$ Sobolev-norms.
More specifically:
Let:
$$\mathcal {OI}:= \left\{ s\in C^{\infty} (E):\; \nabla _{\frac{\p}{\p \bar z}} s= 0 \text{ and } g_\CC(s,s)=0\right\}$$
and:
$$\mathcal I_{S^1}  :=\left\{ s\in C^{\infty} (E\mid_{S^1}):\;  g_\CC(s,s)=0\right\}$$

\begin{proposition}\label{holomorphic-isotropic}
Let $H$ and $g_\CC$ be as above.
There exists a surjective map:
$$T:\mathcal {OI}\to \mathcal I_{S^1}$$
Furthermore we can find a set of boundary data such that the corresponding counter-images via $T$ are holomorphic isotropic sections $s$ of $E$ such that:
\begin{enumerate}
\item $|s(0)|=1$

\item $|s|_H^2\leq \kappa$
\item $|\p s|_H^2\leq\frac{ \kappa}{R^2}$
\item $|s(z)|\geq \frac{1}{2}$ if $z\in B_{\frac{R}{2\sqrt{\kappa}}} (0)$ 
\item more generally $|s(z)|\geq 1-a$ if $z\in B_{\frac{a \,R}{\sqrt{\kappa}}} (0)$, for $a\in (0,1).$

 \end{enumerate}
\end{proposition}

\begin{proof}

The map $T$ is simply given by restricting a given section $s$ to the boundary:
$$T(s):= s\mid_\gamma$$
where $\gamma := \p  D\simeq S^1$.
In order to show surjectivity, we first show we can solve the $\bar \partial$-problem for any boundary condition:
 \begin{equation}\label{dbar-problem}\left\{ \begin{aligned}\frac{\p} {\p\bar z} s&=0 \text{ in } D \\
       {s_\mid }_{\gamma}&=\chi
\end{aligned}\right.
\end{equation}
where $\chi \in C^{\infty}(D)$ is to
be specified later (here $\gamma := \partial D$). That we can
solve this equation with any boundary condition is guaranteed by
the fact that $E$ is holomorphically trivial on $D$ (cf. Theorem Y pg. 211 in 
\cite{gr}), and therefore the problem reduces to the
~1-dimensional Cauchy-Riemann problem for functions on $D$, which
can be solved using the Cauchy integral formula:
$$s_i(\zeta) = \frac{1}{2\pi \sqrt{-1}}\int_\gamma \frac{\chi_i(z)}{z-\zeta}\,dz$$
where --after identifying $E$ with $D\times \CC^n$-- $s_i$ are the components of $s$ and $\chi= (\chi _1,\cdots, \chi _n)$.
We next show that if $\chi \in  \mathcal I_{S^1}$ then the solution of \eqref{dbar-problem} is in $\mathcal {OI}$. In order to make sure that $s$ be isotropic only in terms of the boundary
 condition we make use of the fact that if $s$ is a holomorphic section of $E$, then:
 $$\bar \p g_\CC(s,s)=0$$ 
 i.e., $ g_\CC(s,s)$ is holomorphic, which descends immediately from the fact that:
 $$\bar \p g_\CC(s,s)= g_\CC(\nabla _{ \frac{\p}{\p \bar z} }  s,s)+g_\CC(s,\nabla _{ \frac{\p}{\p \bar z} }  s)=0$$
Indeed, we can choose $\chi= (\chi _1, \cdots , \chi _n)$ such that $ g_\CC\mid _\gamma (\chi
 , \chi )=0$, therefore by analytic continuation, also $s$, the solution to the Cauchy problem $\bar \p s=0$ and $s\mid _{\p D}=\chi$, is such
 that $ g_\CC(s,s)=0$; i.e., $s$ is an {\it isotropic} section. The fact that this can be done is a simple consequence of the fact that we can choose $ \chi = \alpha + \sqrt{-1} \beta$ with $\ \alpha$ and $\\beta$ sections of $\p D\times \RR^n$ such that:
$$\Vert   \alpha\Vert _g^2 = \Vert   \beta \Vert _g^2 \text{    and    } \langle \alpha ,  \beta \rangle _g=0$$
where $g$ is the real form of the Hermitian metric $H$.
  This shows the surjectivity of the map $T$.

 We now show that we can choose the boundary data $\chi$ so that (1)--(4) hold.
 
 Part (1)  and (2) now follow from choosing the $\chi_i$'s accordingly as follows. In the global trivialization chosen:
 \begin{equation}\label{normofs}\begin{aligned} &|s|_H^2(z)= H_{i\bar j} (z)s^i \bar s^j\\&= H_{i\bar j} (z)\left( \frac{1}{2\pi \sqrt{-1}}\int_\gamma \frac{\chi_i(z)}{z-\zeta}\,dz\right) \overline{ \left( \frac{1}{2\pi \sqrt{-1}}\int_\gamma \frac{\chi_i(z)}{z-\zeta}\,dz\right)}\end{aligned}\end{equation}
 and since in polar coordinates centered at $0\in D$ (i.e., $z= re^{i\theta}$) there holds $\frac{dz}{z} \mid_\gamma = \sqrt{-1} d\theta$, in order for (1) to hold it suffices to require:
 \begin{equation} \label{norm1}\sum _{i,j=1}^n\, H_{i\bar j}(0) \left(\frac{1}{2\pi} \, \int _0^{2\pi} \, \chi _i (\theta)\, d\theta\right) \left( \frac{1}{2\pi} \,\int _0^{2\pi} \, \bar \chi _j (\theta)\, d\theta \right)=1\end{equation}
We show the existence of such $\chi$ as follows. We choose a holomorphic frame $e_i$ for $E$ on (an open set containing) $D$ such that at $0$:
$$H(e_i, e_j)(0)= H_{i\bar j }(0)= \delta _{ij}$$
In this frame we choose a smooth section of $E\mid_{\p D}$ of the form $\tilde \chi=(\tilde \chi_1,\cdots, \tilde \chi_n)$ and we have chosen $\tilde \chi_i$ such that:
$$g_\CC(\tilde \chi,\tilde \chi)= g_{\CC ij}\,\tilde \chi_i\tilde \chi_j=0.$$

\noindent
This is tantamount to choosing $\tilde \chi =\tilde \alpha + \sqrt{-1} \tilde \beta$ with $\tilde \alpha$ and $\tilde \beta$ sections of $\p D\times \RR^n$ such that:
$$\Vert \tilde  \alpha\Vert _g^2 = \Vert \tilde  \beta \Vert _g^2 \text{    and    } \langle \tilde \alpha , \tilde \beta \rangle _g=0$$
where $g$ is the real form of the Hermitian metric $H$ (that is $H(v,w)= g_\CC (v,\bar w)$ where $g_\CC$ is the bilinear extension of $g$).

Set $\chi _i:=  e^{\sqrt{-1}\lambda \theta}\, \tilde \chi_i$, $\alpha:=  e^{\sqrt{-1}\lambda \theta}\, \tilde \alpha$, $\beta:=  e^{\sqrt{-1}\lambda \theta}\, \tilde \beta$  and $\chi:=(\chi_i,\cdots , \chi _n)$. Clearly one still has that:

$$\Vert \alpha\Vert _g^2 =\Vert \tilde  \alpha\Vert _g^2 = \Vert \tilde  \beta \Vert _g^2  = \Vert \beta \Vert _g^2 \text{    and    } \langle \alpha , \beta \rangle _g=0$$
or equivalently: 
$$g_\CC(\tilde \chi,\tilde \chi)=g_\CC( \chi,\chi)=0 \text {   and   } \Vert \chi \Vert_H^2=  \Vert \tilde \chi \Vert_H^2.$$
%
One can readily show that one can have chosen $\tilde \chi=\tilde \alpha +\sqrt{-1} \tilde \beta$ satisfying the following:
\begin{itemize}
\item the necessary conditions for $\tilde \chi$ (hence for $\chi$, whence for $\sigma$) to be {\it isotropic}:
\begin{equation}\label{boundarynorm} \Vert \tilde  \alpha\Vert _g^2 = \Vert \tilde \beta \Vert _g^2 \text{    and    } \langle \tilde \alpha , \tilde \beta \rangle _g=0\end{equation}
\item the Euclidean norm of $\tilde \chi$ (whence the one of $\chi$) satisfies\begin{footnote}{This can be simply achieved by replacing $\chi$ with $\frac{\chi}{ \sum _{i=1}^n \vert\tilde \chi _i\vert^2}$, if necessary}\end{footnote}:
\begin{equation} \label{euclideanboubndartnorm}\sum _{i=1}^n \vert\tilde \chi _i\vert^2=1\end{equation}
%
\end{itemize}
We now prove that we can choose $\chi$ so that item $(1)$ of the Proposition holds. Since:
$$I_\lambda :=\sum _{i,j=1}^n\, H_{i\bar j}(0) \left(\frac{1}{2\pi} \, \int _0^{2\pi} \,  e^{\sqrt{-1}\lambda \theta}\tilde  \chi _i (\theta)\, d\theta\right) \left( \frac{1}{2\pi} \,\int _0^{2\pi} \,  e^{-\sqrt{-1}\lambda \theta}\overline{\tilde \chi _j (\theta)}\, d\theta \right)$$
is a continuous expression in $\lambda$ and since $\lim _{\lambda\to \infty} I_\lambda = 0$ (this is a particular case of the fact that the Fourier transform is an automorphsm on the space of Schwarz functions, cf. \cite{stein}) it follows that there exists a choice of $\lambda$ for which:
 $$\sum _{i,j=1}^n\, H_{i\bar j}(0) \left(\frac{1}{2\pi} \, \int _0^{2\pi} \,  e^{\sqrt{-1}\lambda \theta} \chi _i (\theta)\, d\theta\right) \left( \frac{1}{2\pi} \,\int _0^{2\pi} \,  e^{-\sqrt{-1}\lambda \theta}\bar  \chi _j (\theta)\, d\theta \right)=1$$
 \noindent
 That is to say equation \eqref{norm1} holds.
Then if $s$ is the solution of the Dirichlet problem with $\chi:=(\chi_i,\cdots , \chi _n)$ as boundary condition (i.e., $s=T^{-1}(\chi)$) one can easily check that there is a choice of $\lambda$ such that:
$$\begin{aligned} &|s|_H^2(0)= H_{i\bar j} (0)s^i \bar s^j=|T^{-1}(\chi)|_H^2(0)\\&= \sum _{i,j=1}^n\, H_{i\bar j}(0) \left(\frac{1}{2\pi} \, \int _0^{2\pi} \, \chi _i (\theta)\, d\theta\right) \left( \frac{1}{2\pi} \,\int _0^{2\pi} \, \bar \chi _j (\theta)\, d\theta \right)=1\end{aligned}$$
which settles $(1)$.

We next prove that (2) holds in two different ways. 
In the first proof we simply exploit the assumption that $ H\leq \kappa \, H_0$. In the second (which we only sketch) one employs that the curvature of $H$ is bounded (as consequence of the fact that by assumption $ H\leq \kappa \, H_0$ and $\Vert \nabla H \Vert _{H_0}, ||\nabla^2 H||_{H_0}<\kappa $).
 One easily proves (e.g., choosing a frame at any given $p\in D$ where $H_{i\bar j}(p)= \delta _{ij}$, $\p H (p)= \bar \p H (p)=0$ and $\frac{\p ^2H_{i\bar j}}{\p z\p \bar z}  (p) = -R(H)_{i\bar j}$) the following Bochner type formula:
 \begin{equation} \label{bochner} \p \bar \p \vert  s \vert _H^2= - R(H)_{i\bar j} s^i \bar s^j + \vert \p s\vert _H^2\end{equation}
 We can now proceed by observing that on the one hand:
 \begin{equation}\label{easy-ineq} \vert  s \vert _H^2\leq\kappa \;\vert  s \vert _{H_0}^2 \end{equation}
 and that on the other hand $- R(H_0)_{i\bar j} =0$, therefore, using equation \eqref{bochner} applied to $ \vert  s \vert _{H_0}^2$ and applying the maximum principle to the differential inequality:
 $$\p \bar \p  \vert  s \vert _{H_0}^2= \vert \p s\vert _{H_0}^2,$$
  yields:
 $$\sup _D  \vert  s \vert _{H_0}^2= \sup _{\p D}  \vert  s \vert _{H_0}^2$$
 whence (coupled with eq. \eqref{easy-ineq}):
\begin{equation}\label{boundary-control}\vert  s \vert _H^2\leq\kappa \; \vert  s \vert _{H_0}^2\leq \kappa \; \sup _{\p D}  \vert  s \vert _{H_0}^2\leq \kappa \end{equation}
 having used that, by construction $ \sup _{\p D}  \vert  s \vert _{H}^2=1$ (as a consequence of eq. \eqref{euclideanboubndartnorm}).
%
Since by equation \eqref{boundarynorm} :
 \begin{equation}\label{boundarynorm-2} \Vert   \alpha\Vert _g^2 = \Vert  \beta \Vert _g^2=\frac{1}{2}\end{equation}
 \noindent
 thus, using eq. \eqref{normofs} and eq. \eqref{boundary-control}:
$$\begin{aligned}& |s|_H^2(\zeta)\\&\leq 2\, \sup _{\zeta \in\p D} \left( H_{i\bar j} (\zeta)\left( \frac{1}{2\pi \sqrt{-1}}\int_\gamma \frac{\chi_i(z)}{z-\zeta}\,dz\right) \overline{ \left( \frac{1}{2\pi \sqrt{-1}}\int_\gamma \frac{\chi_j(z)}{z-\zeta}\,dz\right)}\right)\\&\leq 2\;  \sup _{\p D} \left( \Vert \alpha \Vert_H^2+  \Vert \beta \Vert_H^2\right) =2\end{aligned}$$

 Item (3) is a consequence of the inequality (itself a consequence of the Cauchy-integral formula):
 $$|\frac{\p^k}{\p z^k}s_i| (x_0)\leq \frac{k!}{R^k} \sup _{\p B(0, R)} |\chi|$$
 for any ball $ B(x_0, R)$. Therefore on such a ball:
 $$\Vert \p s\Vert _{H_0}^2\leq \sum _{i=1}^n \;   \sup  _{\p B(0, R)}|\frac{\p}{\p z}s_i| ^2 \leq \frac{1}{R^2}\,  \sum _{i=1}^n \;  \sup _{\p B(0, R)} |\chi|^2= \frac{1}{R^2}\,  \sup  _{\p B(0, R)} \Vert \chi \Vert _{H_0}^2$$
 In particular, making use of inequality \eqref{boundary-control}:
 \begin{equation} \label{boundondels}|\p s|_H^2\leq \kappa |\p s|_{H_0}^2\leq \kappa \; \frac{1}{R^2}\,  \sup _{\p D} ||\chi||_{H_0}^2\leq   \frac{\kappa }{R^2}\, \end{equation}
 where in the next to last inequality we employ the fact that by choice (see footnote on page 13) $\chi =e^{-\phi} v$ for some constant vector $v$.

 We now notice that:
 $$\vert \p |s|_H^2\vert \leq 2 \vert \p s\vert \, \vert s\vert  \;\; \text{   and   } \;\; \p |s|_H^2 = 2|s|_H \p |s|_H$$
 combining which yields:
 $$\p |s|_H \leq |\p s|_H$$
 which, coupled with eq. \eqref{boundondels} yields:
 $$\p |s|_H \leq\frac{\sqrt{\kappa} }{R}$$
 Finally, having bounded the gradient of $|s|_H$, and therefore bounded the Lipshitz constant of $|s|$, item (5) (whence item (4)) follows.

 The second proof (which we only sketch) is based on an argument similar to the one employed in the proof of Proposition \ref{tweaking}, which  we can use to show that we can conformally change $H$ so that $R(e^{-\psi} H)_{i\bar j}<0$ with $\psi$ such that $\Vert \psi\Vert _{C^{k,\alpha}} <C$ (cf. \eqref{calderon} in  Proposition \ref{tweaking}), Bochner formula (eq.  \eqref{bochner}), which holds for any holomorphic section $s$) and the maximum principle.
 \end{proof}
\begin{remark}
This proposition is similar in spirit to Lemma 2.1 in \cite{CF}, except here we make sure we can find solutions with controlled norm (specifically bounded away from zero) at least on  a half the disk. Also here we exploit directly the holomorphic triviality of holomorphic vector bundles on the disk (more generally polydisks) rather than solving the Riemann-Hilbert problem for the coupled $\bar \p$ operator.
\end{remark}
\subsection{The model example: isotropic holomorphic Gaussian sections}\label{model-gaussian}

In this section, we make use of a slight modification of Donaldson's technique to construct Gaussian holomorphic sections of the trivial bundle on the ball of radius $R$ in $\CC$. 

We diverge from Donaldson's treatment a bit as we need to make sure that the "local" construction produces {\bf isotropic} holomorphic sections.

Let then $B_R$ be the ball (i.e., disk) of radius $R>2$ in $\CC$ (in the application we will take $N=1$) with the standard flat metric and standard K\"ahler form:
 $\Omega_{0}=\sqrt{-1}\,  dz\wedge d\bar z$. Let $F$ be the trivial rank $n$ holomorphic vector bundle $F= B_R \times \CC^n$ with metric conformal to the flat Hermitian metric by the factor $\exp(-\vert z\vert^{2}/2)$.
 The ~1-form:
 $$A_k:= \frac{k}{2}\,\left(  z\,d\bar z - \bar z \, dz\right)$$
 gives rise to a diagonal connection on $F$:
 $$A_{K} :=\bigoplus _{i=1}^n A_{k_i}$$
for any multi-index $K= \left( k_1,\cdots , k_n\right)$ and we will assume $k_i\geq 0$.

 The curvature $F_{A_K}$ of $A_k$ is simply $dA_K=-\sqrt{-1}\,k\, \Omega _0$ and there is (up to constant rescaling) only one Hermitian metric $H_k$ on $B_R \times \CC$  compatible with $A_k$ (i.e., the metric $H_K$ whose curvature is $F_{A_K}$): the metric $H_k=e^{-\frac{k\,|z|^2}{2}}\, h_0$ where $h_0$ is the standard flat Hermitian metric on $B_R \times \CC.$ 
 
 Observe that the connection $A_{K}$ gives rise to a $\bar \partial _{A_{K}}$ operator on $B_R \times \CC^n$ as follows:
 $$\bar \partial _{A_{K}} (s_1,\cdots , s_n) = ( \bar \partial _{A_{k_1}} s_1, \cdots , \bar \partial _{A_{k_n}} s_n)$$
 where $\bar \partial _{A_{k_i}} s= \bar \p + A_{k_i}^{0,1} s$ (here $A_{k_i}^{0,1}$ indicates the $(0,1)$-part of $A_{k_i}$).
 and also to a Hermitian metric on $F$ whose curvature is: $$dA_{K}= -\sqrt{-1}\,\bigoplus _{i=1}^n k_i\, \Omega _0,$$
 namely the diagonal metric: 
 \begin{equation} \label{hermitian} H_{K,\underline C}=\bigoplus _{i=1}^n \,C_i\; e^{-\frac{k_i\,|z|^2}{2}}\, h_0\end{equation}
 where $\underline C=(C_1,\cdots, C_n)$ with $C_i\in\RR$, $C_i>0$ and $h_0 (s,s) = s\, \bar s$ on $B_R\times \CC.$
  \begin{definition} \label{isotropic-standard} Let:
 $$g_{K,\uC,\CC}(\sigma, \sigma) :=\sum \, C_i\;\exp(-k_i\vert z\vert^{2}/2)  \; \sigma _i  \sigma _i$$
 where $\sigma= (\sigma_1,\cdots, \sigma _n)$.  A section of $F$, $\sigma$ say,  is said to be {\bf isotropic} if:
$$g_{K,\uC,\CC}(\sigma, \sigma)=0.$$
 \end{definition}
 \begin{remark} 
 Clearly $g_{K,\uC,\CC}$ is the $\CC$-linear extension of the metric on $B_R\times \RR^n$ whose Hermitian extension is $H_K$.
 \end{remark}
 
 In what follows we will find it convenient to switch between different representations (corresponding to different gauges): one which we might view as having fixed the Hermitian metric $H_0$ and having represented the complex structure on $F$ as $\bar \p _{A}$, or equivalently we will fix the complex structure (which in our trivializing charts is the standard one) and then consider the Hermitian metric on $F$ as being given by $H_{n,k}$ (the equivalence of these views can be seen as an incarnation of the Poincare- Lelong formula for $\p \bar \p \log |s|_H^2$ for a Hermitian metric $H$ and a holomorphic section $s$).
 We are now ready to prove:
 \begin{lemma}\label{holo-isotrop-standard}
 
 On the (necessarily) trivial rank $n$ holomorphic bundle $F$ over $B_R$ endowed with the metric $H_{K,\underline C}$ of equation \eqref{hermitian}, the ball of radius $R$.
 And assume that:
 \begin{itemize}
 \item $K= (k_1,\cdots, k_n)$ with $k_1\geq k_2 \geq \cdots k_{n-1}\geq k_n\geq 0$
 \end{itemize}
  Let $\kappa:= \max\{ C_i\}>0$. Then there exist $R_0>0$ and a connection $A_{n,K}$ (with associated Hermitian metric $H_{n,K}$) and smooth section $\sigma$ such that:
 \begin{enumerate}
 \item $F_{A_{K}}=dA_{K}+[A_{K},A_{K}]=-\sqrt{-1}\oplus_i \,k_i\,\Omega  $, where $K=(k_1,\cdots, k_n)$
 \item  $\bar \p _{A _{K}}\sigma =0$

 \item $\sigma$ is isotropic: $g_{K,\uC,\CC}(\sigma,\sigma)=0$.
 \item $|\sigma (0)|_{H_{n,K}}=1$

 \item $\vert \sigma \vert _{H_{K, \uC}}=e^{-\frac{k_n|z|^2}{4}}|\sigma |_{H_{0,K}}$, with $|\sigma |_{H_{0,K}}\leq \kappa$
\item $|\sigma(z)|_{H_{n,k}}\geq e^{-k_n\frac{a^2 R}{\kappa}}( 1-a) $ if $z\in B_{\frac{a \,R}{\sqrt{\kappa}}} (0)$.   

\item $\pi <\Vert \sigma \Vert _2<2\pi.$
\item $\Vert \sigma \Vert _{L^2(B_R)}^2 \leq \frac{2{\kappa}}{1-a}  \Vert \sigma \Vert _{L^2(B_\frac{aR}{2\sqrt{\kappa}})}^2$
 \end{enumerate}
 for every $R\geq R_0$.
 \end{lemma}
 \begin{proof}
 Set:
 $$A_K :=\bigoplus _{i=1}^n A_{k_i}$$
 as above. Then clearly (since $[A_K,A_K]=0$):
 $$F_{A_K}=dA_K= \bigoplus dA_{k_i}= -\sqrt{-1}\bigoplus k_i \, \Omega $$

 \noindent
By definition: 
$$H_{K,\uC} = e^{-k_n |z|^2} \, H_{0,K}$$
where:
$$H_{0,K}:=\bigoplus _{i=1}^n \,C_i\; e^{-\frac{(k_i-k_n)\,|z|^2}{2}}\, h_0$$ 
is such that: 
$$  H_{0,K}\leq \kappa H_0$$ 
for $\kappa:= \max\{ C_i\}>0$, since $k_i-k_n\geq 0$ for any $i$.
Therefore, we can appeal to  Proposition \ref{holomorphic-isotropic} and produce a  {\it holomorphic isotropic} (that is isotropic with respect to $e^{k_n \,|z|^2}g_{K,\uC,\CC}$) section $\sigma_0= (\sigma_{0,1},\cdots, \sigma _{0,n})$ such that (as constructed in Proposition \ref{holomorphic-isotropic})\begin{footnote} {Here we tacitly use the fact that the complex structure induced from $\bar \p _{A_{K'}}$-- where $K':= (k_n-k_1,\cdots, k_2-k_1, 0)$-- and the standard holomorphic structure are equivalent }\end{footnote}:
\begin{itemize}
 \item $|\sigma _0(0)|=1$

\item $|\sigma _0|_g^2\leq \kappa$
\item $|\p \sigma _0|_g^2\leq \frac{\kappa }{R}$
\item  $|s(z)|\geq 1-a$ if $z\in B_{\frac{a \,R}{\sqrt{\kappa}}} (0)$.
\end{itemize}

\noindent
We can now rescale the Euclidean metric $g_0$ on $\CC$ if necessary, so that we may assume $k_n=1.$
Next, set:
 $$\sigma :=\exp(-\vert z\vert^{2}/2)  (\sigma_{0,1},\cdots, \sigma _{0,n})$$
 \noindent
which manifestly satisfies items $(5)$ and $(6)$.
 Then $\sigma $ is $A_K$-holomorphic (i.e., item (2) holds):
 \begin{align}&\bar \p _{A_K} \left( \exp(-\vert z\vert^{2}/2)  \sigma_{0,i}\right)= \\&= \bar \p \left( \exp(-\vert z\vert^{2}/2) \right)  \sigma_{0,i}+ A^{(0,1)} \, \exp(-\vert z\vert^{2}/2)  \sigma_{0,i}=0\nonumber\end{align}
for every $i\in \{ 1,\cdots, n\}$,  having used that $\bar \p   \sigma_{0,i}=0$.
 The norm of  $\sigma $ is $\exp(-\vert z\vert^{2}/4)\, |\sigma _0 |^2$ and $\sigma$ is isotropic with respect to $g_\CC$ (i.e., item (3) holds).
 Also:

$$\begin{aligned}& \vert \p _{A_k} \left( \exp(-\vert z\vert^{2}/2)  \sigma_{0,i}\right)\vert= \\&=\vert  \p \left( \exp(-\vert z\vert^{2}/2) \right)  \sigma_{0,i}+\exp(-\vert z\vert^{2}/2) \p ( \sigma_{0,i})+ A^{(1,0)} \, \exp(-\vert z\vert^{2}/2)  \sigma_{0,i}\vert \\ &\leq \vert  \p \left( \exp(-\vert z\vert^{2}/2) \right) \vert \sigma_{0,i}\vert+\vert\exp(-\vert z\vert^{2}/2)\vert \,\vert\p ( \sigma_{0,i})\vert \\&+\vert A^{(1,0)} \vert\,\vert \exp(-\vert z\vert^{2}/2)  \sigma_{0,i}\vert \end{aligned}$$
Number $(8)$ is a consequence of the fact that for $R$ sufficiently large:
$$\int _{  B_{\frac{a \,R}{\sqrt{\kappa}}} (0)} \, \vert \sigma \vert _{H_{K, \uC}}^2 \; dx\wedge dy\geq (1-a)2\pi \int_0 ^{\frac{a \,R}{\sqrt{\kappa}}}  e^{\frac{r^2}{2}}\; d\left(\frac{r^2}{2}\right) \geq \pi (1-a)$$
and:
$$\int _{  B_R(0)} \, \vert \sigma \vert _{H_{K, \uC}}^2 \; dx\wedge dy\leq \kappa  2\pi \int_0 ^R e^{\frac{r^2}{2}}\; d\left(\frac{r^2}{2}\right) \leq 2 \kappa \pi$$
\noindent
The rest is straightforward.  \end{proof}
 \begin{remark}
 It is important to remark that the positivity of the curvature is what produces the Gaussian type holomorphic sections. For instance, on the line bundle $L=D\times \CC$ with Hermitain metric $e^{|z|^2} h_0$ with negative curvature, the holomorphic sections one produces have exponential growth.
 \end{remark} 

\subsection{General facts about holomorphic section of vector bundles}

Next we prove some general well known theorems on holomorphic sections of a Hermitian holomorphic vector bundle on a K\"abler manifold $(N,h)$ with the extra assumption that the background metric $h$ satisfies a lower bound on the Ricci curvature. This will not apply immediately to our context, but it will apply to the context in which we endow the disk with the flat metric (or a rescaled version of the flat metric). We will then be able to effect the control desired merely because the sections we construct have $L^2$-norm in the background metric which is comparable (by a given and definite amount) to the $L^2$-norm with respect to the flat metric. Our proofs follow the lines of \cite{DS}, where they discuss the case of line bundles.
\begin{proposition}\label{controlforholsection}
Let $(E,H)$ be a holomorphic Hermitian vector bundle on a K\"ahler manifold $(N, h)$ and $\sigma$ a holomorphic section of $E$ such that following conditions hold:

$${\rm K}(H)(\p \sigma, \p \sigma)\geq -C _K \qquad \text{ and } \qquad {\rm Ric}(h)\geq -C _R$$

\noindent
where by $K(H)$ we denote the {\it mean curvature} of $H$ (cf. section \ref{curvature})
Then:
\begin{enumerate}
\item  
$ \Vert\sigma \Vert_{L^{\infty}(H)} \leq \kappa _{0} \Vert \sigma \Vert_{L^{2}( H))}  ,\; \Vert \nabla \sigma \Vert_{L^{\infty}( H))} \leq \kappa_{1} \Vert \sigma \Vert_{L^{2}( H))} $
for some uniform constants $\kappa _0$ and $\kappa _1$;
\item  $\vert \sigma (x)\vert\geq 1/4$ at all points $x$ a distance (in the rescaled metric $h_R;= R^2\, h$ on $N$) less than $\min \{ \frac{R}{2}, (4\kappa_{1})^{-1}\}$ from $0$, for some uniform $\kappa _1$ depending only on $C_K$ and $C_R$;

\end{enumerate}
\end{proposition}

\begin{proof}

We produce a uniform derivative estimate first and then use Moser iteration. This produces a uniform estimates since the lower bound on ${\rm Ric}(g)$ entails a uniform control on the Sobolev constant. 
 Let $\nabla^{*}$ and $\db^*$ be calculated with respect to $h$.
First remark that:
$$   \nabla^{*}\nabla s = 2 \db^{*} \db s+ s, $$
therefore in the ball $B_  \frac{R}{2}$ (here the ball is calculated with respect to $g_R$) --where $s$ is holomorphic-- $\nabla^{*}\nabla s = s$ which implies that (in the sense of barriers):
 \begin{equation}\Delta \vert s\vert \leq \vert s\vert, \end{equation} 
 since on the one hand:
 $$\Delta \vert s\vert ^2 = 2\langle \nabla^{*}\nabla s, s\rangle= 2 \vert s\vert^2$$
 and on the other:
  $$\Delta \vert s\vert ^2 =2 \vert s\vert \Delta \vert s\vert+ 2 \vert \nabla \vert s\vert \vert ^2\geq 2 \vert s\vert \Delta \vert s\vert .$$
 Now the bound on the $L^{\infty}$ norm follows from the Moser iteration argument applied to this differential inequality (see \cite{Tian2}). Remark that the Sobolev constant here is uniform because of the lower bound on ${\rm Ric(G_R)}=0$, so the bound obtained from Moser iteration is uniform.

Next we derive the first derivative bound, i.e., the second part in item $(3)$, which in turn implies item $(4)$.  Again we restrict ourselves to the ball  $B_  \frac{R}{2}$, where $s$ is a  holomorphic section, thus   $\db s=0$ and therefore $\nabla s = \partial s$ where $$\partial:\Omega^{p,q}(E)\rightarrow \Omega^{p+1,q}(E)$$ is defined using the connection. Since $\partial^{2}=0$ we have
$$  \Delta_{\partial } \partial s = \partial \Delta_{\partial} s, $$
where $\Delta_{\partial}=\partial^{*}\partial + \partial \partial^{*}$. Then for a holomorphic section $s$ , $\Delta_{\partial} s =\nabla^{*}\nabla s = s$ and
$$   \Delta_{\partial} (\partial s) = \partial s. $$
The Bochner-Kodaira-Nakano formula (cf. Theorem \ref{bochner-kodaira}) involving $\Delta_{\partial}$ and $\nabla^{*}\nabla$ on $\Omega^{1,0}(E)$ is:
$$   \Delta_{\partial} = \nabla^{*}\nabla -1+{\rm K} (H) $$

\noindent
This yields (using that ${\rm R } (H_\psi) \geq \omega _G$): 
$$ \nabla^{*}\nabla =  \Delta_{\partial} +1-{\rm K} (H) \leq  \Delta_{\partial} +1$$
so: 
$$\begin{aligned} \langle \nabla^{*}\nabla (\partial s), \partial s\rangle &=\langle  \Delta_{\partial}  (\partial s), \partial s\rangle +\langle   \partial s, \partial s\rangle - {\rm K} (H)(\partial s,\partial s) \\&= 2\vert \p s \vert ^2  - {\rm K } (H)(\partial s,\partial s)  \leq (C_K+2) \vert\partial s \vert^{2}. \end{aligned}$$
 It follows that $$\Delta \vert \partial s \vert\leq (C_K+2)\,  \vert \partial s \vert, $$
and the Moser argument applies as before. Notice that the constants only depend on the lower bound $C_R$ of the Ricci curvature ${\rm Ric}(h)$, the Sobolev constant  of $h$ and the dimension. Therefore we have shown item $(3)$. Item $(4)$ follows from it, since item $(3)$ bounds the Lipschitz constant of $|s|$. 
 
\end{proof}

\begin{remark}
Remark that by Lemma \ref{rescale}, in our application we can take the K\"ahler metric $h$ on the disk to be the flat metric $dx^2+dy^2$ or a rescaling of it.
\end{remark}
\noindent
With a stronger hypothesis on the structure of the induced metric on the minimal immersion, we can also prove the following, which shall not be used in the proof of the main theorems but is of independent interest.
\begin{theorem}
Let $f:\Sigma \to M$ be a minimal immersion of a compact (not necessarily closed) Riemann surface and assume that the induced metric $g_\Sigma$ on $\Sigma$ satisfies:
$${\rm Ric (g_\Sigma)} \geq -C_\Sigma$$
If the isotropic curvature of $(M,g)$ satisfies $K^{isotr}_\CC\geq C_K$, then there exists a holomorphic  $g$-isotropic section $\sigma$ of $E=f^*TM \otimes \CC\to \Sigma$ such that:

\begin{enumerate}

\item  $1\leq \Vert \sigma \Vert_{L^2}\leq \frac{11}{5}\pi$;
\item $\bar \p \sigma=0$ on the ball of radius $\frac{R}{2}$ (in the rescaled metric);
\item  $ \Vert\sigma \Vert_{L^{\infty}(H)} \leq \kappa _{0} \Vert \sigma \Vert_{L^{2}( H)} ,\;\Vert \nabla \sigma \Vert_{L^{\infty}( H)} \leq \kappa_{1} \Vert \sigma \Vert_{L^{2}( H)}  $
for some uniform constants $\kappa _0$ and $\kappa _1$;
\item  $\vert \sigma (x)\vert\geq 1/4$ at all points $x$ a distance (in the rescaled metric $g_R$) less than $\min \{ \frac{R}{2}, (4\kappa_{1})^{-1}\}$ from $0$, for some uniform $\kappa _1$ depending only on $C_\Sigma$ 
\end{enumerate}

\end{theorem}
\begin{proof}
According to  Lemma \ref{almast standard} we can find a rescaling $\Psi _R: B_R \to B$  so that (for $R$ sufficiently large) we can get $E_R \to B_R$ to satisfy the hypotheses of Theorem \ref{loca-model}. 

Let $\Delta_{\db}= \db^{*}\db+ \db \db^{*}$, with adjoints defined using the rescaled metric $g_R:= \Phi _R^* g=  \Phi _R^* \left( \lambda (dx^2+dy^2)\right)$, then for all $\phi$
\begin{equation}  \langle \Delta_{\db}^{-1} \phi, \phi\rangle_{g_R} \leq \frac{R}{R-C}  \Vert \phi \Vert^{2}_{L^{2}(g_R)}\end{equation}
In fact by Kodaira-Nakano (after using the natural isometry $\Omega ^{0,q} \otimes E\simeq \Omega^{n,q} \otimes E \otimes K_\Sigma ^*$) :
\begin{equation}   \Delta_{\db}= (\nabla^{(0,1)})^{*} \nabla^{(0,1)} + {\rm K(H,g_{\Sigma,R})}+ {\rm Ric} (g_\Sigma) \end{equation}
where $g_{\Sigma,R}$ is the result of rescaling $g_\Sigma$ and $ {\rm K(H,g_{\Sigma,R})}_{\alpha \bar \beta}= g_{\Sigma ,R}^{i\bar j} R(H)_{i\bar j \alpha \bar \beta}$.
Whence:
 $$  \Delta_{\db}\geq \frac{C_K-C_\Sigma}{R}$$ 
in the operator sense (since $ {\rm Ric} (g_R)\geq -\frac{C_\Sigma}{R}$ and $ {\rm K(H,g_{\Sigma,R})}\geq\frac{C_K}{R}$, the latter when restricted to isotropic sections).

Thus, if we set $s= \sigma- \tau$ where $\tau= \db^{*} \Delta_{\db}^{-1} \db \sigma$ then clearly $\db s=0$. Also:
$$\Vert \tau\Vert_{L^{2,R}} = \langle \Delta_{\db}^{-1} \db \sigma, \db\ \db^{*}\Delta_{\db}^{-1}\db \sigma \rangle = \langle \Delta_{\db}^{-1} \db \sigma, \db \sigma \rangle, $$
since $\db\ \db \sigma=0$. Thus
\begin{equation} \Vert \tau\Vert_{L^{2,R}}\leq \sqrt{\frac{R}{R-C_K} } \Vert \db \sigma \Vert_{L^{2,R}} \end{equation}
Hence in particular, for $R$ sufficiently big:
 $$\Vert s\Vert_{L^{2,R}}\leq \Vert \sigma\Vert_{L^{2,R}}+\Vert \tau\Vert_{L^{2,R}}\leq \frac{11}{10} 2\pi $$

Also, since on the ball of radius $\frac{R}{2}$, $\sigma$ is holomorphic, it follows that $s=\sigma$ on $B_{\frac{R}{2}}$ (since $\bar \p \sigma =0$ implies $\tau=0$) and therefore $s$ is isotropic there since $\sigma$ is. Since $s$ is holomorphic everywhere so is $g_\CC(s,s)$ but since s is isotropic on $B_{\frac{R}{2}}$, which is equivalent to $g_\CC(s,s)=0$, it follows by analytic continuation that $s$ is isotropic everywhere.

The rest is Proposition \ref{controlforholsection}.
\end{proof}

\subsection{Making the complex structure and the bundle almost standard }

Let $B\subset \CC $ be the unit ball and let $E\to D$ a holomorphic vector bundle endowed with a Hermitian metric $H$. 

Since $E\to D$ is a holomorphic line bundle we can infer the fact that $E$ is holomorphically trivial on $D$ (cf.
\cite{gr}), that is to say there is an isomorphism:
$$\phi:E\to D\times \CC^n$$

Clearly the map $\phi$ above is not an isometry of bundles. Nonetheless in order to apply Donaldson's philosophy we merely need to construct an {\it almost isometry}.

More precisely, fix an integer $k$ and rescale the background metric $g$ on $\Omega$ by a factor of $R^2=k$ and consider the Hermitian metric $H^{k}$ on $E$-- here $H^{k}$ is calculated by diagonalizing $H$ and then taking the $k$-th power.
To set this up formally, we introduce the following notation. Denote by:

$$\Psi_R : D_R \to B$$
the standard dilation by $R$: 
$$\Psi_R (z)= w_0 +\frac{z}{R}$$
and let $E_k:=\Psi _R ^* E$, $g_R:= R^2\,\Psi _R^*g=  \lambda (w_0 + \frac{z}{R}) |dz|^2$ (where $g= \lambda \,|dw|^2$) and $H_k:= \Psi_R^*\left( H^k\right)$ and also $\phi_R:=  \Psi_R^*\phi.$ 

Observe that for the curvature form of $H_R$ one has:
$$\begin{aligned}\Theta (H_R)_i^s&= H_R^{s\bar j}R(H_R)_{i\bar j}(z)\,  dz\wedge d\bar z\\&= \phi ^*\left( \Theta (H)_i^s\left(w_0+ \frac{z}{R}\right) \right) \,dz\wedge d\bar z\end{aligned}$$
where $R(H)_{i\bar j} := R(H)(e_i, \bar e_j, \frac{\p}{\p z}, \frac{\p}{\p\bar z})$ are the components of the curvature of $H$ and $\Theta (H)$ is the curvature ~2-form.
%
%
It is thus manifest that measured with respect to the rescaled metrics:, for any $\epsilon >0$ there exists $R_0$ such that:
$$ \Vert \Theta (H_R) - \Theta_0\Vert <\epsilon$$
for any $R\geq R_0$, where $\Theta _0= \Theta (H) (w_0) \;  dz\wedge d\bar z$.

Let $\Lambda:=(\lambda_1, \cdots , \lambda _n)\in \RR_+^n$ an n-tuple of positive numbers and consider the endomorphism of $\CC^n$:
$$\Lambda \, Id_{\CC^n} : \CC^n \to \CC^n \;\;\;\;\;\; \Lambda \, Id_{\CC^n} (z_1, \cdots, z_n) = (\lambda _1z_1, \cdots , \lambda _n z_n)$$
 
Observe that, for any $\epsilon >0$, there exists an $R$ sufficiently big, such that: 
\begin{equation} \label{curvatureclose}\Vert K - \Lambda \, Id_{\CC^n}  \Vert_{C^{\infty}}< \epsilon \end{equation}
where $\Omega _0$ denotes the (K\"ahler form associated to the) flat metric on $\CC^n$.
The following fact is now obvious:
\begin{lemma}\label{almast standard}
Up to a constant endomorphism of $E$, the Hermitian bundle $(E, H_k)$ is nearly isometric, via $\Phi_R$, to the bundle $B_R \times \CC^n$ endowed with the Hermitian metric  (defined up to a constant endomorphism of $D\times \CC^n$):
 $$H_\Lambda =\bigoplus _{i=1}^n \; e^{-\frac{\lambda_i\,|z|^2}{2}}\, h_0$$ whose curvature form is:
$$\Omega_\Lambda:= \oplus _i \lambda _i \Omega _0$$ 
\end{lemma}
\begin{proof}
First observe that by equation \eqref{curvatureclose} we may assume that in some scale $R_{i\bar j} (H)$ and  and $ \Lambda \, Id_{\CC^n}$ are $\epsilon$-close.
Since we are on a Riemann surface, the curvature of the metric $H$ takes the form (cf. formula \eqref{hermitiancurvature-dim1})
$$R_{i \bar j  }:= R(e_i, \bar e_j, \frac{\p}{\p z}, \frac{\p}{\p\bar z})= - \frac{\p ^2 H_{i\bar j}}{\p z \p \bar z}+ \frac{\p H_{i\bar t}} {\p z} H^{s\bar t}  \frac{\p H_{s\bar j}} {\p \bar z}$$
By elliptic regularity, given a function $K_{i\bar j}$ and a metric on the boundary $h_{i\bar j} \in C^{k,\alpha} (\p D)$, there exists a metric $H_{i\bar j}\in C^{k+2,\alpha}(D)$ with  $K_{i\bar j}$ as curvature, since:
\begin{equation}\left\{ \begin{aligned} -& \frac{\p ^2 H_{i\bar j}}{\p z \p \bar z}+ \frac{\p H_{i\bar t}} {\p z} H^{s\bar t}  \frac{\p H_{s\bar j}} {\p \bar z}= K_{i\bar j}\\ & H_{i\bar j} \mid _{\p D} = h_{i\bar j} \end{aligned} \right. \end{equation}
is an elliptic equation. It is easy to see that $H_{i\bar j}$ must be a metric (i.e., positive definite). In fact, its determinant $H:= \det (H_ {i\bar j})$ satisfies the elliptic equation:
$$- \frac{\p ^2\log (H)}{\p z \p \bar z}= H^{i\bar j} K_{i\bar j}$$
where $H^{i\bar j}$ is the inverse of $H_{i\bar j}$. Therefore, by elliptic regularity, $H:= \det (H_ {i\bar j})$ never vanishes, whence by the connectedness of the domain $\Omega$, the eigenvalues of $H_{i\bar j}$ are positive at every point, since they are so at the boundary.

Furthermore, by the maximum principle, two such solutions differ by a constant endomorphism of $D\times \CC^n$.
Therefore, up to composing with said endomorphism:
$$\Vert H - H _\Lambda\Vert<\epsilon .$$
\end{proof}

\noindent
We can now prove:
 \begin{proposition}\label{holo-isotrop}
Let $D\subset \CC$ be the unit disk endowed with the flat metric $g_0=dx^2+dy^2$. Let $(E,H)$ be a holomorphic Hermitian bundle over $D$ with associated connection ~1-from $A.$ 

Assume that $E=F\otimes \CC$ for some real bundle $F$ {\begin{footnote} {This is automatic on the disk $D$}\end{footnote}} and assume further that $F$ is endowed with an inner product structure $g$ and that its complex bilinear extension  $g_\CC$ is such that $H(v,w):= g_\CC (v,\bar w)$. 
Let $w_0\in D$ any point and $r>0$ such that $B_r(w_0)\subset D$.

If the curvature of $H$ is positive on isotropic two planes, then there exist an $R_0>0$ and a smooth section $\sigma$ of $E$ such that, if $H_R: = k H$ and $R:=\sqrt{k}$, for $r\geq R_0$:
  \begin{enumerate}

 \item  $\bar \p _{A}\sigma =0$, i.e. $\sigma$ is holomorphic;

 \item $\sigma$ is isotropic: $g_\CC(\sigma,\sigma)=0$.
 \item $|\sigma (w_0)|_{H_R}=1$

\item $|\sigma(z)|_{H_R}\geq e^{-\frac{|z|^2}{4}} \, (1-a) $ if $z\in B_{\frac{a}{2}} (w_0)\subset D$.
\item $\pi \leq \Vert \sigma \Vert _{H_k,L^2(B_R)}\leq 2\pi$
\item $\Vert \sigma \Vert _{H_R,L^2(B_R)}^2 \leq \frac{2{\kappa}}{1-a}  \Vert \sigma \Vert _{H_R,L^2(B_\frac{aR}{2\sqrt{\kappa}})}^2$

 \end{enumerate}
 The same holds true if one merely assumes that $K_{iso}^{\CC}\geq -C$, but the constants will depend on $C$.
 \end{proposition}
 \begin{proof}
%

 Let $w_0\in D$ any point in the interior.
We choose $k=R^2$ by requiring that after rescaling, using the map $\Phi _R$ defined at the beginning of this section, on $D_R$ we can achieve:
  $$\sup _{D_R} \Vert \Theta (H_k)- \Theta (H_k)(w_0)\Vert <\epsilon$$
  for any given $\epsilon >0$.
Therefore, according to Lemma \ref{almast standard} we have:
$$\sup _{D_R} \Vert H_k- H_\Lambda\Vert<\epsilon$$
\noindent
where $\Lambda = (\lambda _1,\cdots, \lambda _n)$,
$$ \Theta (H_R)(w_0)=\bigoplus _i \lambda _i \Omega _0 $$
and
$$H_\Lambda =\bigoplus _{i=1}^n \; e^{-\frac{\lambda_i\,|z|^2}{2}}\, h_0$$ 
 
 \noindent 
 Next remark that the complex structure on $E\to D$ is rigid, up to equivalence, thanks to  Theorem Y pg. 211 in 
\cite{gr}, as already remarked in Proposition \ref{holomorphic-isotropic}.
 Since $H_\Lambda \leq H_0$, where $H_0$ is the Euclidean Hermitian metric on $E\simeq D\times \CC^n$ we can choose $R$ sufficiently large so that 
 $$\frac{100}{81} \,H_0\leq H_R \leq \frac{100}{81}\,H_0.$$

 Whence, according to  Lemma \ref{holo-isotrop-standard}, we can produce a holomorphic {\it isotropic} section $\sigma_0= (\sigma_{0,1},\cdots, \sigma _{0,n})$ of $(E,H_R)$ satisfying (here the constant $\kappa$ appearing in the hypotheses of Lemma \ref{holomorphic-isotropic} is $\kappa=\frac{100}{81}$):
 \begin{itemize}
\item $|\sigma_0(0)|=1$

\item $|\sigma_0|_{H_R}^2\leq 2$
\item $|\p \sigma_0|_{H_R}^2\leq 2$
\item $|\sigma_0(z)|_{H_{n,k}}\geq e^{-k_n\frac{a^2 R}{\kappa}}( 1-a) $ if $z\in B_{\frac{a \,R}{\sqrt{\kappa}}} (0)$.   

\item $\pi <\Vert \sigma _0\Vert _2<2\pi.$
\item $\Vert \sigma _0\Vert _{L^2(B_R)}^2 \leq \frac{2{\kappa}}{1-a}  \Vert \sigma \Vert _{L^2(B_\frac{aR}{2\sqrt{\kappa}})}^2$


 \end{itemize}
 
 \noindent
 Where the $L^2$-norms are taken with respect to volume form $dx\wedge dy$.
 We now consider the bundle $S$ whose sheaf of sections is:
 
 $$\left\{ \alpha \in C^{\infty}(E):\; \bar \p _A \alpha =0 \text{ and } \; K(H)(\alpha, \alpha ) \geq 0\right\}$$
 which is clearly non-empty since isotropic sections belong to it by assumption. 
 
 It is now easy to show that after pulling back via $\psi _R: D_R\to D$ and rescaling the metric $g_R=R^2\psi^*g$, the second fundamental form of $S$ becomes negligible and therefore, by Lemma \ref{positivityofquotients}, we may assume that $(S,H_S)$ has non-negative curvature (up to rescaling) (alternatively, we can endow $S$ with the complex structure determined by the connection $\nabla _E$ of $E$ and then argue that, for $R$ sufficiently large, the induced sub-bundle complex structure and this complex structure are arbitrarily close).
  
 But in fact we can do better. Since every vector bundle splits holomorphically on the disk $\Delta$, thanks to Theorem Y pg. 211 in 
\cite{gr}, it is also the case that every short exact sequence of vector bundles must split. Thus, we can view $S$ as being, in a natural way, a quotient of $E$: $E\to S\to 0$. 
Thus, appealing to Lemma \ref{positivityofquotients}  yields that $S$ has non-negative curvature operator, with respect to the induced Hermitian metric and induced Hermitian structure (so that one does not have to work with errors).
 We then make use of Theorem \ref{loca-model} (see below) to argue that we can confuse the induced complex structure with the complex structure for which $S$ has non-negative curvature operator.

%

Whichever way, we can next apply the arguments of Lemma \ref{holo-isotrop-standard} to the holomorphic section $e^{-\lambda _n |z|^2} \sigma _0$ to prove (1)-(3), (5) and (6).
 We now achieve the uniform estimate in item (4) by applying Proposition \ref{controlforholsection} (with $h=g_0= dx^2+dy^2$) to control uniformly the $L^{\infty}$-norm of $\Vert \nabla s\Vert$ and therefore the Lipschitz constant of $|\sigma|$.

 \end{proof}
Next we need (in the spirit of Property (H) in \cite{DS}):

\begin{proposition}\label{main-standard-2}
 In the same assumptions as above, there exists an $R>0$, a smooth {\bf isotropic} section, $s$ of $E$ such that:
\begin{enumerate}
 \item  $ \pi <\Vert s\Vert_{L^{2}}<  2\pi;$

\item $\vert s(0)\vert =1;$

 \item For any smooth section $\tau$ of $E$ over a neighborhood of $\overline{D}$ we have
$$   \vert \tau(0)\vert \leq C \left( \Vert \bar \p  \tau \Vert_{L^{p}(D)} +  \Vert \tau \Vert_{L^{2}(D)}\right);$$
where the volume form is the Euclidian one.
\item   $\Vert \bar \p  s\Vert_{L^{2}}< \frac{3}{R} \left(\Vert   s \Vert _{L^2(B_R)} + e^{-\frac{R^2}{2}} 2\pi R\right);$
\item $|\sigma(z)|_{H_R}\geq e^{-\frac{|z|^2}{4}} \, (1-a) $ if $z\in B_{\frac{a}{2}} (w_0)\subset D$.

\item $\Vert \sigma \Vert _{H_R,L^2(B_R)}^2 \leq \frac{2{\kappa}}{1-a}  \Vert \sigma \Vert _{H_R,L^2(B_\frac{aR}{2\sqrt{\kappa}})}^2$


\end{enumerate}
\end{proposition}
\begin{proof}

For simplicity of notation we set $B\equiv B_R$.
First of all, item (3) holds for a $C$ independent of $R.$ In fact, given $B'\subset B$ some interior domain containing $0$, the standard elliptic estimate    \begin{equation}    \Vert \tau \Vert_{L^{p}_{1}(B')} \leq C_e \left( \Vert \db \tau \Vert_{L^{p}(B)} + \Vert \tau \Vert_{L^{2}(B')} \right), \end{equation}
coupled with the Sobolev inequality 
$$   \vert \tau(u_{*})\vert \leq C_S \Vert \tau\Vert_{L^{p}_{1}(B')}. $$
yield (3).

 Let $\eta_{R}=\eta _R(\vert z\vert)$ be a cut-off function:
 
 $$\eta _R(\vert z\vert)= \left\{ \begin{aligned} &1 \text{ if }\vert z\vert \leq R/2\\&0 \text{ if }  \vert z\vert \geq \frac{9R}{10} \end{aligned}\right. .$$ 
 such that:
 $$\vert \eta _R' \vert\leq \frac{3}{R}$$
Let $\sigma$ be as in Lemma \ref{holo-isotrop-standard}.  Define: 
 $$s=\eta_{R}\sigma.$$
  Then we have $\db s= (\db \eta_{R}) \sigma$, therefore:
  \begin{equation}\Vert \db  s\Vert = \Vert \db \eta_{R}\, \sigma\Vert = \vert\db  \eta_{R}\vert \, \Vert \sigma\Vert\end{equation}
  
  \noindent
  whence (using that $\Vert \sigma\Vert _{L^2(B_R)}=e^{-\frac{R^2}{2}} \, \Vert  \sigma_0 \Vert _{L^2(B_R)}$ and that $ \Vert  \sigma_0 \Vert _{L^2(B_R)}\leq 2\kappa \pi$):
  $$\begin{aligned}& \Vert \db  s\Vert _{L^2(B_R)} \leq \frac{3}{R} \Vert   \sigma \Vert _{L^2(B_R)} \leq   \frac{3}{R} \left(\Vert   \sigma \Vert _{L^2(B_\frac{R}{2})} + \Vert  \sigma \Vert _{L^2(B_R\setminus B_\frac{R}{2})} \right)\\ &\leq \frac{3}{R} \left(\Vert   \sigma \Vert _{L^2(B_\frac{R}{2})} + e^{-\frac{R^2}{2}} \, \Vert  \sigma_0 \Vert _{L^2(B_R)}  \right)\leq\frac{3}{R} \left(\Vert   \sigma \Vert _{L^2(B_\frac{R}{2})} + e^{-\frac{R^2}{2}} 2\kappa \pi R\right) \\&\leq \frac{3}{R} \left(\Vert   s \Vert _{L^2(B_R)} + e^{-\frac{R^2}{2}} 2\kappa \pi R\right) \end{aligned}$$
  which proves item (4). In the second inequality of the series of inequalities above we used the fact that:
  $$ \Vert   \sigma \Vert _{L^2(B_R)} ^2=\Vert   \sigma \Vert _{L^2(B_\frac{R}{2})} ^2+ \Vert  \sigma \Vert _{L^2(B_R\setminus B_\frac{R}{2})} ^2$$
 and that if $x, y\geq 0$ then $\sqrt{x^2+y^2} \leq x+y$.
  
  Next, one easily verifies that:
  $$\Vert s\Vert_{L^{2}}= 2\pi-\delta \qquad \Vert s\Vert_{L^{2}(G)}= 2\pi-\delta'$$
for some small, $\delta, \delta ' >0$-- which proves item (1)-- and also, trivially:
 $$\vert s(0)\vert =1.$$ 
 This proves item (2).
   

\end{proof}

The next Proposition is due to Donaldson and Song (cf. \cite{DS} )
\begin{proposition} \label{openness} The properties (1)-(6) in Proposition \ref{main-standard-2}
are open with respect to variations in $(g,J,A)$ (for fixed $(B,D,0, F)$) and the topology of convergence in $C^{0}$ on compact subsets of $U$.
\end{proposition}
%

We are now ready to prove the main result of this section. First we define the following sets:

$$\mathcal M(\epsilon,C):=\left\{ (M,g) : \;  K^{isotr}_{\CC} (M) \geq \epsilon ^{-2}\right\}$$
and 
$$\mathcal K :=\left\{ f: D\to (M,g): \;  \begin{aligned} & f\text{ is a minimal {\it proper} immersion and } \\&(M,g)\in \mathcal M(\epsilon,C)\end{aligned}\right\}$$
where $D\subset \RR ^2$ is the unit disc.

\begin{theorem}\label{loca-model}
Suppose that $(D_R,D,0, F:=D_R\times \CC^n)$ are as above and the datum $g_{0},J_{0}, A_{0}, H_0$ satisfies properties (1)-(6) in Proposition \ref{main-standard-2}. Also assume that the Hermitian metric $H_0$ is such that $H_0\leq H_e$ where $H_e$ is the standard flat Hermitian metric. Then there is some $\epsilon_0>0$ such that if  $(D,f,M,g)$ is in $ {\mathcal K}$ and  we can find $R>0$,  a scaling $\Phi _R:D_R\rightarrow B$ with $\Phi _R(0) =0$ and a bundle isomorphism
$\hat{\Phi }_R: F \rightarrow E:= f^*(TM\otimes \CC)$  such that 
$$\Vert \hat{\Phi }_R^{*}(J)-J_{0}\Vert_{U}, \Vert \Phi _R^{*}(g)-g_{0}\Vert_{U}, \Vert \Phi _R^{*}(A) - A_0 \Vert_{U} \leq \epsilon, \Vert \Phi _R ^* H- H_0\Vert <\epsilon $$
with $\epsilon\leq \epsilon _0$
 then there is a smooth section $s$ of $E$ such that:
 
 \begin{enumerate}
 
\item  $ \pi <\Vert s\Vert_{L^{2}}<  2\pi;$

\item $\vert s(0)\vert =1;$

 \item For any smooth section $\tau$ of $F$ over a neighborhood of $\overline{D}$ we have
$$   \vert \tau(0)\vert \leq C \left( \Vert \bar \p  \tau \Vert_{L^{p}(D)} +  \Vert \tau \Vert_{L^{2}(D)}\right);$$

\item   $\Vert \bar \p  s\Vert_{L^{2}}< \frac{3}{R} \left(\Vert   s \Vert _{L^2(B_R)} + e^{-\frac{R^2}{2}} 2\pi R\right)<\frac{9}{R} \Vert s\Vert_{L^2} ;$

 \item $|\sigma(z)|_{H_R}\geq e^{-\frac{|z|^2}{4}} \, (1-a) $ if $z\in B_{\frac{a}{2}} (w_0)\subset D$.
\item $\Vert \sigma \Vert _{H_R,L^2(B_R)}^2 \leq \frac{2{\kappa}}{1-a}  \Vert \sigma \Vert _{H_R,L^2(B_\frac{aR}{2\sqrt{\kappa}})}^2$

 \end{enumerate}
\end{theorem}

\begin{proof}
Note that the Hermitian metric $H$ on $E$ corresponding to the datum $(A,J,g)$ satisfies:
$$\Vert H- H_ 0\Vert _{H_0}<\epsilon$$
where $H_0$ is the Hermitian metric $H$ on $E$ corresponding to the datum $(A_0,J_0,g_0)$. Since by assumption $H_0\leq H_e$, we may assume $H\leq 2H_e$. Hence, according to Proposition \ref{holomorphic-isotropic}, we can find a holomorphic isotropic section $s$ satisfying:
\begin{enumerate}
\item $|s(0)|=1$

\item $|s|_H^2\leq 2$
\item $|\p s|_H^2\leq 2$
\item $|s(z)|\geq \frac{1}{2}$ if $z\in B_{\frac{R}{2}} (0)$ 
\item $\pi \leq \Vert s\Vert _2^2\leq 2\pi$

 \end{enumerate}
\noindent
The rest of the proof goes like the proof of Proposition \ref{holo-isotrop}.
Items (1)-(4) are a straightforward consequence of Proposition \ref{openness} and Proposition \ref{main-standard-2}. Item (5) follows from Proposition \ref{controlforholsection}
since the $L^2$-norm of with respect to $\Phi^*_Rg$ (where $g$ is the metric induced from the embedding) is comparable to the $L^2$-norm calculated with respect to the $\Phi_R^* G$, where $G=dx^2+dy^2$.

 Let us elaborate item (4) a bit.
We have:
$$\Vert \bar \p  s\Vert_{L^{2}}< \frac{3}{R} \left(\Vert   s \Vert _{L^2(B_R)} + e^{-\frac{R^2}{2}} 2\pi R\right);$$
therefore, using that $\Vert s\Vert>\pi$ and that $2x e^{-\frac{x^2}{2}} \leq \frac{2}{\sqrt{e}}< 2$ (so that $e^{-\frac{R^2}{2}} 2\pi R< 2 e^{-\frac{R^2}{2}}  R \, \Vert s\Vert <2\Vert s \Vert $):
$$\Vert \bar \p  s\Vert_{L^{2}}\leq \frac{9}{R} \Vert s\Vert_{L^2}$$

\end{proof}

\subsection{The destabilizing section and the Main theorem}

In this section we fix the flat metric $G:= dx^2+dy^2$ on $D$ and the metric $f^*g$ on $E$ (here we abuse notation in writing $f^*g$, meaning the Hermitian metric induced by $H:=f^*g$ on $E:= f^*TM \otimes \CC$). We will also denote by $\Delta$ the flat Laplacian on $D$ (i.e., the Laplacian on functions associated to $G$) 

In order to prove the Main theorem we will need:

\begin{proposition}\label{destabsection} 
For any point $p\in f(D)$ such that $r:={\rm dist}_{f(D)} (p ,\p f(D))$ there exists a compactly supported $g$-isotropic section $s=s_p= \eta \sigma$ of $E\to D$, where $\sigma$ is isotropic and holomorphic and $\eta$ is compactly supported smooth function, such that:
\begin{enumerate}

 \item The support of $s $ is contained in the ball of radius $r$ centered at $p$: $B_r(p)$;
\item  $ \pi <\Vert s \Vert_{L^{2}}<  2\pi;$


\item   $\Vert \bar \p   s\Vert_{L^{2}}^2< \frac{9}{r^2} \Vert   s \Vert _{L^2(B_r)}^2 ;$

 \item  $\frac{81\, n\, \pi}{4} \,\Vert s \Vert_{L^2(B_{\frac{r}{2}})}^2 \geq  \Vert \sigma \Vert_{L^2(B_r)}^2$
\end{enumerate}
\end{proposition}

\begin{proof}

We first rescale (and translate) so that $B_r(p)$ is the unit ball centered at $0$. 
According to  Lemma \ref{almast standard} we can find a rescaling $\Psi _k: B_k \to B$  so that (for $k$ sufficiently large) we can get $E_k \to B_k$ to satisfy the hypotheses of Theorem \ref{loca-model}. We also choose $R$ sufficiently large so that $g_R:= R^2\Psi _k ^* g= \lambda \left(x_0 + \frac{x}{R}, y_0 + \frac{y}{R}\right) (dx^2+dy^2)$, where $w_0=(x_0,y_0),$ is such that:
$$ \frac{1}{2} (dx^2+ dy^2) \leq g_R\leq 2(dx^2+ dy^2).$$
 We then get, by a straightforward application of Proposition \ref{holo-isotrop}, a holomorphic section $\sigma_k$ of $(E,H_k)$ satisfying (in the rescaled metric):

 \begin{itemize}
 
\item  $ \pi <\Vert \sigma\Vert_{L^{2},g_0,H_k}<  2\pi;$

\item $\vert \sigma(0)\vert_{H_k} =1;$

\item $\pi <\Vert \sigma \Vert _{L^{2},g_0,H_k}<2\pi.$
\item $\Vert \sigma \Vert _{L^2(B_R), g_0,H_k}^2 \leq \frac{2{\kappa}}{1-a}  \Vert \sigma \Vert _{L^2(B_\frac{9 aR}{10}),g_0,H_k}^2$, with $\kappa =\frac{100}{91}$

%
%

 \item  $|\sigma(z)|\geq e^{-\lambda |z|^2} (1-a)$ if $z\in B_{\frac{9 a \,R}{10}} (0)$  
 \end{itemize}

\noindent
Where for emphasis we have indicated by $\Vert \beta \Vert_{L^{2},g_0,H_k}$ the $L^2$-norm of a section $\beta$ calculated using the Euclidean metric $g_0$ and the Hermitian metric $H_k$.

We now set $\sigma_k:=( \sigma^\frac{1}{k}_1,\cdots , \sigma^\frac{1}{k}_n)$ --which is defined in the interior of $B_{\frac{9}{10}}$ by the last item above, since it shows there are no zero in the interior of the ball $B_{\frac{9}{10}}$-- and we observe that:

\begin{equation}\label{root-ineq} \Vert \sigma \Vert _{H_ k}^\frac{2}{k}\leq \Vert \sigma_k\Vert _H^2\leq n \,  \Vert \sigma \Vert _{H_ k}^\frac{2}{k}\end{equation}
as one can easily prove by diagonalizing $H$ (hence $H_k$) at a point (and using that $ \left( \sum _i a_i^2\right)^\frac{1}{k}\leq \sum _i a_i ^\frac{2}{k}\leq n \left( \sum _i a_i^2\right)^\frac{1}{k}$ for any $(a_1,\cdots ,a_n)\in \RR^n$) .

Next, let  $\eta_{r}=\eta _r(\vert z\vert)$ be a standard cut-off function:
 
 $$\eta _r(\vert z\vert)= \left\{ \begin{aligned} &1 \text{ if }\vert z\vert \leq r/2\\&0 \text{ if }  \vert z\vert \geq \frac{9r}{10} \end{aligned}\right. .$$ 
 such that:
 $$\vert \eta _r' \vert\leq \frac{3}{r}$$
Let $\sigma$ be as above and define: 
 $$s=\eta_{r}\sigma.$$
 
  Then we have $\db s= (\db \eta_{R}) \sigma$, therefore:
  \begin{equation}\Vert \db  s\Vert = \Vert \db \eta_{R}\, \sigma\Vert = \vert\db  \eta_{R}\vert \, \Vert \sigma\Vert\end{equation}
  
  \noindent
  which proves item $(4)$.
  As for item $(5)$, equation \eqref{root-ineq} and the fact that $\Vert \sigma \Vert _{L^2(B_R), g_0,H_k}^2 \leq \frac{2{\kappa}}{1-a}  \Vert \sigma \Vert _{L^2(B_\frac{aR}{2\sqrt{\kappa}}),g_0,H_k}^2$ yield:
  \begin{equation} \label{L2-estimate-1}\begin{aligned}& \Vert \sigma_k \Vert _{L^2(B_R), g_R,H}^2\leq n\, \Vert \sigma \Vert _{L^2(B_R), g_R,H_k}^2\leq  2 n \Vert \sigma \Vert _{L^2(B_R), g_0,H_k}^2\\&\leq 2n\, \frac{200}{91(1-a)}  \Vert \sigma \Vert _{L^2(B_\frac{9aR}{10}),g_0,H_k}^2\leq 4n\frac{200}{91(1-a)}  \Vert \sigma \Vert _{L^2(B_\frac{9aR}{10}),g_R,H_k}^2
 \end{aligned}\end{equation}
 Next observe that:
  \begin{equation} \label{L2-estimate-2} \begin{aligned}&\Vert \sigma \Vert _{L^2(B_\frac{9aR}{10}),g_R,H_k}^2\\&= \left( \Vert \sigma \Vert _{L^2(B_\frac{9aR}{10}),g_R,H_k}^2\right) ^{\frac{k-1}{k}} \,  \Vert \sigma \Vert _{L^2(B_\frac{9aR}{10}),g_R,H_k}^\frac{2}{k}\\\leq & \left( 2\pi \right)^{\frac{k-1}{k}} \,  \Vert \sigma \Vert _{L^2(B_\frac{9aR}{10}),g_R,H_k}^\frac{2}{k}\leq 2\pi  \,  \Vert \sigma \Vert _{L^2(B_\frac{9aR}{10}),g_R,H_k}^\frac{2}{k} \end{aligned}\end{equation}

 Therefore putting together equations \eqref{L2-estimate-1} and \eqref{L2-estimate-2}:
 $$\Vert \sigma_k \Vert _{L^2(B_R), g_R,H}^2\leq \frac{800}{91 (1-a)}\,2 n\, \pi \Vert \sigma \Vert _{L^2(B_\frac{9aR}{10}),g_R,H_k}^\frac{2}{k}$$
 Noticing that $\frac{800}{91}<9$ and using once again eq. \eqref{root-ineq}, we infer:
 $$ \Vert \sigma_k \Vert _{L^2(B_R), g_R,H}^2<\frac{18n\pi}{1-a}  \Vert \sigma _k \Vert _{L^2(B_\frac{9aR}{10}),g_R,H}^2$$
 Thus, choosing $a=\frac{5}{9}$, we have:
 $$ \Vert \sigma_k \Vert _{L^2(B_R), g_R,H}^2<\frac{81\,n\,\pi}{4} \Vert \sigma _k \Vert _{L^2(B_\frac{R}{2}),g_R,H}^2$$
 Scaling back to $g$, we prove item (4).
  
%
 
\end{proof}

\begin{remark}
It is important to notice that, since we don't really apply the full H\"ormander technique to find holomorphic sections -- as we only care about almost holomorphic ones-- and since we ultimately simply apply rescalings, we can afford to achieve point 4 in Proposition \ref{destabsection} with constants (and radii) that are independent of the full sectional curvature.
This could also be achieved by insisting that $\p \sigma_R$ be isotropic too. Ultimately this point is irrelevant for the topological application on the fundamental group, but it is important to prove Gromov's conjecture as originally stated.
\end{remark}
We can now prove our main estimate:
\begin{theorem}\label{testsection}
Let $f:D \to M$ be a stable, minimal (possibly branched) immersion. Then for every point $p\in D$ there exists a smooth {\it
  isotropic}  section
$\sigma =\sigma _p $ of
$E$ which is perpendicular to $\frac{\p f}{\p z}$, and a constant $C$ such that:
\begin{equation} \label{testsections} \frac{\il{D}|\nabla _{\frac{\partial}{\partial \bar z}}\sigma |^2
dV}{\il{D}|\sigma |^2 dV}\leq C\frac{1}{r^2} .\end{equation}
where $r:= \rm{dist}_{D} (p, \p D).$
Furthermore, the constant $C$ is computable and it can be chosen to be $\frac{9^3}{4}\pi\, n $.

\end{theorem}
\begin{proof}
Let us fix notation first. We will view $f^*g$ simultaneously as inducing a Hermitian metric on the vector bundle $E$ and as inducing a (necessarily) K\"ahler metric on $D$. In its 
incarnation as the latter, we will write it as $g= \lambda \,G= \lambda \; (dx^2+dy^2)$.
As before, we denote the Hermitian metric induced on $E$ by $H$.

We will in fact consider the vector bundle $N:= \nu _f \otimes \CC$ (here $\nu_f$ is the normal bundle of the map $f$, cf. section \ref{secondvariation}) with the complex structure compatible with $\nabla ^\perp$ (whose existence is guaranteed by Koszul-Malgrange theorem). Let $H_n$ be Hhe hermitian metric induced by the quotient map $E\to N$. 

Since the curvature of $g$ is {\it positive} on totally isotropic ~2-planes, we infer that for any isotropic section $s$ of $N$ (so that $s$ and $\frac{\p}{\p z}$ are independent) we have that:
$$K(H)(\frac{\p}{\p z}, s, \frac{\p}{\p \bar z}, s)>0$$
therefore a straightforward application of Lemma \ref{positivityofquotients} yields that for any {\it isotropic} section $s$ of $N$:
$$K(H_n)(\frac{\p}{\p z}, s, \frac{\p}{\p \bar z}, s)>0$$

Given a point $p\in D$ at distance $r$ from the boundary, as in the hypothesis, we consider the ball $B_r(p)$ centered at $p$ of radius $r.$ We then rescale the induced metric $g= \lambda \; (dx^2+dy^2)$ on $D$ so that $B_r(p)$ becomes the unit disk $D$ centered at the origin. It now suffices to prove that the inequality \eqref{testsections} in the theorem holds with $r=1$.


Whence, according to  Proposition \ref{destabsection}, we can find a smooth $g_{\CC}$-isotropic section $\sigma$ of $E\to B$ such that:
\begin{enumerate}

 \item The support of $s $ is contained in the ball of radius $r$ centered at $p$: $B_r(p)$;
\item  $ \pi <\Vert s \Vert_{L^{2}}<  2\pi;$


\item   $\Vert \bar \p   s\Vert_{L^{2}}^2< \frac{9}{r^2} \Vert   s \Vert _{L^2(B_r)}^2 ;$

 \item  $\frac{81\,n\,\pi}{4} \,\Vert s \Vert_{L^2(B_{\frac{r}{2}})}^2 \geq  \Vert \sigma \Vert_{L^2(B_r)}^2$
\end{enumerate}
Note that item (2) and (4) imply that:
$$ \frac{\il{D}|\nabla  ^\perp_{\frac{\partial}{\partial \bar z}}\sigma |_{H}^2
dV_g}{\il{D}|\sigma |_{H}^2 dV_g}\leq \frac{9^3\,n\,\pi}{4r^2}$$

%
\end{proof}

\noindent
Finally we remark that with Theorem \ref{main2}-- which we just proved-- in hand, the proof of Theorem \ref{main} is a mere application of the techniques of Gromov-Lawson in \cite{gl}, and in particular of the implication that Theorem 10.2 therein implies Theorem 10.7.

\end{document}